\documentclass{amsart}

\usepackage{bigints}
\usepackage{bbm}
\usepackage{geometry}
\usepackage{amssymb}
\usepackage{enumerate}
\usepackage{mathrsfs}
\usepackage{hyperref}
\usepackage[justification=centering, skip=20pt]{caption}
\usepackage[all,cmtip]{xy}
\usepackage{graphicx}
\usepackage{dynkin-diagrams}
\usepackage{setspace}

\newtheorem{innercustomthm}{Theorem}
\newenvironment{customthm}[1]
  {\renewcommand\theinnercustomthm{#1}\innercustomthm}
  {\endinnercustomthm}

\hypersetup{
    colorlinks,
    citecolor=black,
    filecolor=black,
    linkcolor=black,
    urlcolor=black
    }

\newcommand{\R}[0]{\mathbb{R}}
\newcommand{\Z}[0]{\mathbb{Z}}

\renewcommand{\t}[1]{\textup{#1}}

\newtheorem{thm}{Theorem}[section]

\newtheorem{prop}[thm]{Proposition}
\newtheorem{lem}[thm]{Lemma}

\newtheorem*{defn*}{Definition}
\newtheorem{defn}[thm]{Definition}

\newtheorem{rem}{Remark}

\let\emptyset\varnothing

\numberwithin{equation}{section}

\usepackage{etoolbox}

\makeatletter
\patchcmd{\@settitle}{\uppercasenonmath\@title}{}{}{}
\patchcmd{\@setauthors}{\MakeUppercase}{}{}{}
\patchcmd{\section}{\scshape}{}{}{}
\makeatother

\makeatletter
\@namedef{subjclassname@2020}{\textup{2020} Mathematics Subject Classification}
\makeatother

\title[Bounds of restriction of characters to submanifolds]{Bounds of restriction of characters to submanifolds}

\author[Y. Zhang]{Yunfeng Zhang\footnote{Address: Department of Mathematical Sciences, University of Cincinnati, Cincinnati, OH 45221-0025\\ Email: yunfengzhang108@gmail.com}}


\begin{document}

\onehalfspacing

\begin{abstract}
A fruitful approach to studying the concentration of Laplace--Beltrami eigenfunctions on a compact manifold, as the eigenvalue tends to infinity, is to bound their restriction to submanifolds. In this paper, we adopt this approach in the setting of compact Lie groups and provide sharp restriction bounds for general Laplace--Beltrami eigenfunctions, as well as for important special cases such as sums of matrix coefficients and, in particular, characters of irreducible representations.

We prove sharp 
asymptotic $L^p$ bounds for the restriction of general Laplace--Beltrami eigenfunctions to maximal flats and all of their submanifolds, for all $p \geq 2$. Furthermore, we establish sharp asymptotic $L^p$ bounds for the restriction of characters to maximal tori and all of their submanifolds for all $p>0$, and to torus-generated conjugation-invariant submanifolds for all $p \geq 2$. We also obtain sharp $L^p$ bounds for the restriction of general sums of matrix coefficients to maximal flats and all of their submanifolds, for all $p \geq 2$.

\end{abstract}

\maketitle

\tableofcontents

\section{Introduction}

In this paper, we continue the study of concentration of general Laplace--Beltrami eigenfunctions on a compact Lie group as well as important special cases such as sums of matrix coefficients and in particular characters of irreducible representations of the group. Let $U$ denote a compact connected Lie group of dimension $d$ and rank $r$, equipped with the bi-invariant Riemannian metric uniquely determined up to scaling by the Killing form. Throughout, we assume for simplicity that $U$ is simple; the extension of our results to compact semisimple groups is straightforward. 
Let $\Delta$ be the Laplace--Beltrami operator. For an eigenfunction $f$, we write 
\[
  \Delta f = -N^2 f, \qquad N>1,
\]
and state our bounds in terms of the parameter $N$.

Let $L^p(U)$ denote the Lebesgue spaces associated with the Riemannian volume form on $U$. 
A natural way to measure the concentration of Laplace--Beltrami eigenfunctions is via their $L^p$ norms.
In \cite{Zha23} we studied $L^p$ estimates of Laplace--Beltrami eigenfunctions on a compact Lie group. In particular, for $r\geq 4$, we established 
\begin{align*}
\|f\|_{L^p(U)}\leq C_{\varepsilon} N^{\frac{d-2}{2}-\frac{d}{p}+\varepsilon}\|f\|_{L^2(U)}
\end{align*}
for all $p>\frac{2dr}{dr-2(d+r)}$. 
Similar to those eigenfunction bounds on tori as established in \cite{Bou93,BD15}, this serves as a power saving bound over the fundamental $L^p$ bounds of Sogge on a general compact Riemannian manifold \cite{Sog88}.

Initiated by the work \cite{Tat98} of Tataru, 
\cite{Rez04} of Reznikov, and the landmark work \cite{BGT07} of Burq--G\'erard--Tzvetkov, there has appeared a great deal of interesting work on another way of quantifying concentration of eigenfunctions, namely, to establish $L^p$ bounds of restriction of eigenfunctions to submanifolds; see \cite{Sar08, Bou09, Hu09, HT12, BR12, CS14, SZ14, BR15, Che15, Mar15, Mar16, Mar16', Zha17, XZ17, Bla18, Hez18, HZ21, WZ21, Ali22, EP22, Dem23, DL23, Par23, BP24, GMX24, Hou24, HWZ24} for such restriction bounds under various kinds of geometric and arithmetic assumptions on the base manifolds as well as their submanifolds, and we also refer to Chapter 12 of \cite{Zel17} for a recent survey. 
Obtaining $L^p$ estimates for the restriction of eigenfunctions to submanifolds can be more delicate than on the whole manifold, as it requires finer control of their pointwise behavior along the submanifold. 
The following fundamental restriction bounds for general compact manifolds were established in \cite{BGT07,Hu09}.

\begin{customthm}{A}\label{BGT}
Let $M$ be a compact smooth Riemannian manifold of dimension $d$, and let $S$ be a compact smooth submanifold of dimension $k$. Let $L^p(M)$, $L^p(S)$ be the Lebesgue spaces associated to the volume measure on $M$ and $S$ respectively as induced from the Riemannian metric.  
Let $\Delta$ be the Laplace--Beltrami operator on $M$. Then for any eigenfunction $\varphi$ on $M$ such that $\Delta\varphi=-N^2\varphi$, $N>1$, we have 
$$\|\varphi\|_{L^p(S)}\leq C N^{\rho(k,d)}\|\varphi\|_{L^2(M)},$$
where 
\begin{align*}
\rho(d-1,d)&=\left\{\begin{array}{ll}
\frac{d-1}{2}-\frac{d-1}{p}, & \text{ if }\frac{2d}{d-1}\leq p\leq\infty,\\
\frac{d-1}{4}-\frac{d-2}{2p}, & \text{ if }2\leq p\leq \frac{2d}{d-1},
\end{array}\right.\\
\rho(d-2,d)&=\frac{d-1}{2}-\frac{d-2}{p},  \text{ if }2<p\leq \infty, \\
\rho(k,d)&=\frac{d-1}{2}-\frac{k}{p},  \text{ if }1\leq k\leq d-3.
\end{align*}
If $p=2$ and $k=d-2$, we have 
\begin{align}\label{logloss}
\|\varphi\|_{L^2(S)}\leq C N^{\frac{1}{2}}(\log N)^{\frac{1}{2}}\|\varphi\|_{L^2(M)}.
\end{align}
Moreover, all estimates are sharp, except for the log loss\footnote{The log loss in \eqref{logloss} is expected to be eliminated, and this indeed has been done by Chen--Sogge in \cite{CS14} for geodesics lying in a compact manifold of dimension 3, and by Wang--Zhang in \cite{WZ21} for curves with non-vanishing geodesic curvature also in a compact manifold of dimension 3, and for totally geodesic submanifolds of codimension 2 in a compact manifold of any dimension $\geq 3$.  }, if $M$ is picked to be the standard spheres. 
\end{customthm}

The intention of the current paper is to obtain improvement over the general restriction bounds in the above theorem for compact Lie groups. We will consider two classes of submanifolds of the compact Lie group $U$. 
The first class consists of maximal flats in $U$ and all of their submanifolds. A maximal flat in $U$ is defined as a totally geodesic submanifold of
sectional curvature zero of maximal dimension in $U$, or equivalently, a left (or right) translate of any maximal torus of $U$. As a typical example of a submanifold of a maximal flat, any geodesic in $U$ lies in a maximal flat as a one-dimensional submanifold.  
We will first prove the following $L^p$ restriction bounds of characters to such submanifolds. Note that it always holds $$\|\chi\|_{L^2(U)}=1$$ for any character $\chi$. 

\begin{thm}\label{characterrestriction}
Let $\chi$ be the character of an irreducible representation of $U$ 
such that $\Delta \chi=-N^2\chi$, $N>1$. Let $S$ be a compact smooth $k$-dimensional submanifold of a maximal torus $T$ of $U$, $k=0,1,2,\ldots,r$. Let $p_0=0$ and $p_r={2r}/{(d-r)}$. For $k=1,2,\ldots,r-1$, let $p_k$ be as given in Table \ref{criticalexponent}, so that ${(d-r)}/{2}-{k}/{p_k}>0$ and $0<p_k< 1$. Then
$$\|\chi\|_{L^p(S)}\leq C\cdot 
\left\{
\begin{array}{lll}
 N^{\frac{d-r}{2}-\frac{k}{p}}, & \t{ for }p>p_k, 
\\
 N^{\frac{d-r}{2}-\frac{k}{p_k}}(\log N )^{\frac{1}{p_k}}, &\t{ for }p=p_k, \\
N^{\frac{d-r}{2}-\frac{k}{p_k}}, & \t { for }0<p< p_k. 
\end{array}
\right.
$$
Moreover, the above bound is sharp, in the sense that for any $k=0,1,2,\ldots,r$, there exists a compact smooth $k$-dimensional submanifold $S$ of $T$ for which the above bound is saturated by a sequence of characters for all $p>0$. Furthermore, the above bound holds uniformly for all submanifolds $S$ which are translates by elements of $T$ of a fixed compact smooth $k$-dimensional submanifold $S_0$ of $T$, that is 
$$S=xS_0=S_0x=\{xy=yx: \ y\in S_0\},$$
with a constant $C$ independent of $x\in T$. 
\end{thm}

\begin{center}
\begin{table}
{\renewcommand{\arraystretch}{1.2}
\begin{tabular}{|c|c|}
\hline  
Type of $U$ & ${k}/{p_k}$, $k=1,2,\ldots,r-1$ \\
\hline 
$A_r (r\geq 1)$ &  $kr-{(k-1)k}/{2}$\\

\hline 
$B_r (r\geq 2)$ & $2kr-k^2$\\

\hline 
$C_r (r\geq 3)$ & $2kr-k^2$\\

\hline 
$D_r (r\geq 4)$ 
& 
\begin{tabular}{ll}
$2kr-k(k+1)$, & $1\leq k\leq r-4$; \\
$2r(r-1)-6$, & $k=r-3$; \\ 
$2r(r-1)-3$, & $k=r-2$; \\ 
$2r(r-1)-1$, & $k=r-1$.  
\end{tabular}\\

\hline 
$E_6$ & $16, 24, 30, 33, 35$\\

\hline 
$E_7$ &$27, 43, 51, 57, 60, 62$ \\

\hline 
$E_8$ & $57, 84, 100, 108, 114, 117, 119$ \\

\hline 

$F_4$ & $15, 20, 23$\\

\hline 
$G_2$ &  $5$\\

\hline

\end{tabular}
}
\caption{The critical exponents $p_k$}
\label{criticalexponent}

\end{table}
\end{center}

In fact, we will show that for each $k=0,1,\ldots,r$, the above bound is saturated if we pick the $k$-dimensional submanifolds of $T$ to be some of the $k$-dimensional facets of any Weyl alcove in $T$. A (closed) Weyl alcove is a simplex formed as the closure of any connected component of the complement of all the root hyperplanes in the universal cover of the maximal torus. 
Each $k$-dimensional (open) facet of the Weyl alcove lies on finitely many root hyperplanes; in fact, the ones that lie on the largest number of root hyperplanes among all $k$-dimensional facets will be chosen as the submanifolds to saturate the above bound. The semiclassical motivation for these choices begins with the observation that root hyperplanes consist of the focal points of the origin in the maximal torus with respect to the Riemannian geometry of the compact Lie group. Moreover, the multiplicity of each focal point is exactly twice the number of root hyperplanes on which it lies. The Peter--Weyl theorem tells us that any character $\chi$ can be viewed as a component of the Dirac delta function centered at the origin. 
By the quantum--classical correspondence principle, the behavior of $\chi$ should then reflect the classical picture, 
in which particles are emitted from the origin and move along geodesics in directions specified by the frequency of $\chi$ across \emph{all} maximal tori.
These particles reconvene at the focal points of the origin, occurring more frequently when the multiplicity of a focal point is larger—that is, when the number of root hyperplanes containing the focal point is greater.

The above heuristics helps explain the numerology of the exponents of $N$ in the above bound. 
As it turns out, the exponent 
$$\frac{d-r}{2}-\frac{k}{p_k}$$
is the exact number of root hyperplanes containing the chosen $k$-dimensional facet that saturates the corresponding bound. 

\begin{rem}
    As characters themselves are exponential sums on the maximal torus, Theorem \ref{characterrestriction} may prove to have additional values for understanding restriction bounds of general exponential sums (in terms of the scale parameter); see \cite{Dem23, DL23} for restriction bounds of exponential sums along the moment curve. 
\end{rem}

As a consequence of Theorem \ref{characterrestriction}, using Schur's test, we prove the following $L^p$ bounds of restriction of sums of matrix coefficients to maximal flats and their submanifolds.

\begin{thm}\label{jointrestriction}
Let $\psi$ be any sum of matrix coefficients of an irreducible representation of $U$ such that 
$\Delta\psi=-N^2\psi$, $N>1$. Let $S$ be a compact smooth $k$-dimensional submanifold of any maximal flat in $U$, where $k=0,1,2,\ldots,r$. Then for $p\geq 2$, it holds 
$$\|\psi\|_{L^p(S)}\leq C N^{\frac{d-r}{2}-\frac{k}{p}} \|\psi\|_{L^2(U)}$$ 
except when $U=\mathrm{SU}(2)\cong\mathbb{S}^3$, $S$ is (part of) a large circle on $U$, and $p=2$, in which case  we have 
\begin{align*}
\|\psi\|_{L^2(S)}\leq C N^{\frac{1}{2}} (\log N)^{\frac{1}{2}}\|\psi\|_{L^2(U)}.
\end{align*}
The above bounds are all sharp except for the log loss.\footnote{Again, this log loss has been eliminated by Chen--Sogge in \cite{CS14}. }
\end{thm}

For the special case of $S$ being the maximal flats themselves, the above bounds were proved in \cite{Mar16} but under certain regularity assumptions on the spectral parameter of $\psi$. Marshall asked in \cite{Mar16} if these bounds could be established unconditionally and for more general submanifolds. Thus Theorem \ref{jointrestriction} provides positive answers to both of Marshall's questions. The above bounds are also the correct ``convexity'' bounds or ``local'' bounds as coined by Sarnak \cite{Sar08, Mar16} for $L^p$ restriction norms of joint eigenfunctions. 
For future work, it would be interesting to get $L^p$ restriction bounds of general sums of matrix coefficients for $p$ below 2. 

The significance of proving Theorem \ref{jointrestriction} for all sums of matrix coefficients is reflected in its application to restriction bounds of general Laplace--Beltrami eigenfunctions. By a standard estimate of the number of ways of representing an integer by a positive definite integral quadratic form, we offer the following consequence of Theorem \ref{jointrestriction} for compact Lie groups of rank higher than one. 

\begin{thm}\label{gef}
Let $f$ be any eigenfunction on $U$ such that $\Delta f= -N^2f$, $N>1$.  
Let $S$ be a compact smooth $k$-dimensional submanifold of a maximal flat in $U$, where $k=0,1,2,\ldots,r$. Then 
\begin{align}\label{eq: gef}
    \|f\|_{L^p(S)}\leq C N^{\frac{d-2}{2}-\frac{k}{p}} \|f\|_{L^2(U)}
\end{align}
holds for all $p\geq 2$ and $r\geq 5$, and 
\begin{align}\label{eq: gef 234}
    \|f\|_{L^p(S)}\leq C_\varepsilon N^{\frac{d-2}{2}-\frac{k}{p}+\varepsilon} \|f\|_{L^2(U)} 
\end{align}
holds for all $\varepsilon>0$, $p\geq 2$ and $2\leq r\leq 4$.  In particular, for any geodesic segment $\gamma$ in $U$, we have that 
$$\|f\|_{L^p(\gamma)}\leq C N^{\frac{d-2}{2}-\frac{1}{p}} \|f\|_{L^2(U)}$$ 
holds for all $p\geq 2$ and $r\geq 5$, and that 
$$\|f\|_{L^p(\gamma)}\leq C_\varepsilon N^{\frac{d-2}{2}-\frac{1}{p}+\varepsilon} \|f\|_{L^2(U)}$$ 
holds for all $\varepsilon>0$, $p\geq 2$ and $2\leq r\leq 4$.  
\end{thm}

The above bounds are sharp for $r\geq 5$; see Remark \ref{rem: sharp?}. 
For higher rank compact Lie groups, the bounds in Theorem \ref{characterrestriction}, \ref{jointrestriction} and \ref{gef} all improve upon the general restriction bounds in Theorem \ref{BGT} with a power saving. In particular, Theorem \ref{gef} stands as one of the uncommon cases to have restriction bounds with a power saving for {\em general} Laplace--Beltrami eigenfunctions, besides which we only know of the following others: 
\begin{itemize}
\item The $L^p$ ($p\geq 2$) bound of restriction to totally geodesic submanifolds of eigenfunctions on a standard torus as in \cite{BR12, BR15, HZ21}. In fact, our bounds in Theorem \ref{gef} match those bounds in Theorem 9 of \cite{HZ21}. 
\item The $L^\infty$ bound (i.e. of restriction to points) of eigenfunctions $f$ on a general compact globally symmetric space $M$ of dimension $d$ and rank $r$:
$$\|f\|_{L^\infty(M)}\leq C_\varepsilon N^{\frac{d-2}{2}+\varepsilon}\|f\|_{L^2(M)}$$
for $2\leq r\leq 4$ and 
$$\|f\|_{L^\infty(M)}\leq C N^{\frac{d-2}{2}}\|f\|_{L^2(M)}$$
for $r\geq 5$. These bounds can be obtained by applying the Weyl dimension formula to joint eigenfunctions (see \cite{Sar08}) and combining it with the standard estimate for the number of representations of an integer by a positive definite integral quadratic form. 
\item Using sharp $L^p$ bounds for Jacobi polynomials (which serve as the analogue of Theorem \ref{characterrestriction} in the setting of compact rank-one symmetric spaces), we obtain bounds analogous to Theorem \ref{gef} for the restriction of eigenfunctions on any product 
$
M = M_1 \times \cdots \times M_r
$
of compact rank-one symmetric spaces of compact type (in particular, spheres) to submanifolds $S$ of maximal flats. In particular, if each factor $M_i$ has dimension at least $3$, then 
\[
\|f\|_{L^p(S)} \leq C_\varepsilon\, N^{\frac{d-2}{2} - \frac{k}{p} + \varepsilon}\, \|f\|_{L^2(M)}.
\]
We will treat this case in detail in forthcoming work \cite{next}.
\end{itemize}

Next we discuss the second class of submanifolds that will be treated in this paper, defined as follows.

\begin{defn*}
A torus-generated conjugation-invariant submanifold $Y$ of $U$ is defined as the orbit of a compact submanifold $S$ of a facet of a Weyl alcove in a maximal torus of $U$ under the conjugation action.\footnote{As any Weyl alcove of any maximal torus is a fundamental domain for the conjugation action of $U$, the manifold $S$ in the above definition is determined up to conjugation for any torus-generated conjugation-invariant submanifold $Y$.} The rank of $Y$ is defined as the dimension of $S$. 
\end{defn*}

Such $Y$'s are indeed immersed submanifolds of $U$; see Section \ref{invariantsub} for a detailed discussion. A typical example is that of a conjugacy class. To illustrate, Table \ref{tab: submanifolds} presents several representative submanifolds of the rank-$2$ group $\mathrm{SU}(3)$, including two conjugacy classes that are also maximal totally geodesic submanifolds. Each torus-generated conjugation-invariant submanifold of $U$ carries a canonical volume measure induced by the Riemannian metric on $U$. On these submanifolds, we establish the following $L^p$ bounds for characters.

\begin{thm}\label{singular}
Let $Y$ be a torus-generated conjugation-invariant  submanifold of $U$ of dimension $n$ and rank $k$, $k=0,1,2,\ldots,r$.
Let $\chi$ be the character of an irreducible representation of $U$ such that 
$\Delta\chi=-N^2\chi$, $N>0$. Let $p_0=0$ and $p_r={2r}/{(d-r)}$. For $k=1,2,\ldots,r-1$, let $p_k$ be as given in Table \ref{criticalexponent}. Then
$$\|\chi\|_{L^p(Y)}\leq C\cdot 
\left\{
\begin{array}{lll}
 N^{\frac{d-r}{2}-\frac{n}{p}}, & \t{ for }p>2+p_k, 
\\
 N^{\frac{d-r}{2}-\frac{n}{2+p_k}}(\log N )^{\frac{1}{2+p_k}}, &\t{ for }p=2+p_k, \\
N^{\frac{d-r}{2}-\frac{k}{p_k}-\frac{n-k-2k/p_k}{p}}, & \t { for }2\leq p< 2+p_k. 
\end{array}
\right.
$$
Here it always holds that $n-k-2k/p_k\geq 0$. 
Moreover, for each $k=0,1,2,\ldots,r$, there exists a torus-generated conjugation-invariant  submanifold $Y$ of $U$ of dimension $n$ and rank $k$ for which $n-k-2k/p_k=0$ and that the above bound is saturated by a sequence of characters for all $p\geq 2$. 
\end{thm}

Theorem \ref{singular} parallels Theorem \ref{characterrestriction}, and will be proved in a similar manner. In particular, the adapted Weyl integration formula (Lemma \ref{lem: adapted Weyl formula}) transfers integration over torus-generated conjugation-invariant  submanifolds to integration over submanifolds of the facets. The $p=2$ case of the above bounds has additional sharpness, which we record below as a separate theorem.

\begin{thm}\label{singular2}
Let $Y$ be a torus-generated conjugation-invariant  submanifold of $U$ of dimension $n$ and rank $k$, $k=0,1,2,\ldots,r$.
Let $\chi$ be the character of an irreducible representation of $U$ such that 
$\Delta\chi=-N^2\chi$. Then
$$\|\chi\|_{L^2(Y)}\leq CN^{\frac{d-r}{2}-\frac{n-k}{2}}.$$
Moreover, the above bound is sharp in the sense that it is saturated by a sequence of characters whenever $Y$ is the orbit of any facet of the Weyl alcove under the conjugation action. 
\end{thm}

Just as Theorem \ref{characterrestriction}, \ref{jointrestriction} and \ref{gef}, the above two theorems also improve upon the general restriction bounds of Burq--G\'erard--Tzvetkov with a power saving for compact Lie groups of rank higher than one. 

As an important corollary, applying Theorem~\ref{singular} to the set $Y$ of regular points of $U$, and noting that $\|\cdot\|_{L^p(Y)} = \|\cdot\|_{L^p(U)}$, we obtain:

\begin{thm}\label{globalbounds}\footnote{
In \cite{Zha23}, we treated these global character bounds, but the estimate at the kink point and the question of sharpness remained unresolved, as the previous approach was unnecessarily complicated and certain crucial insights were not yet available. These insights include the systematic treatment in Section \ref{peelingrs} of the combinatorics of root hyperplane arrangements, particularly those at the non-origin vertices of the alcove (see Remark \ref{nonorigin}), as well as a simpler argument for applying this combinatorics, as implemented in {\bf Step 5} of the proof of Proposition \ref{keyprop}.
}
Let $\chi$ be the character of an irreducible representation of $U$ 
such that $\Delta\chi=-N^2\chi$. 
Then 
$$\|\chi\|_{L^p(U)}\leq C\cdot 
\left\{
\begin{array}{lll}
 N^{\frac{d-r}{2}-\frac{d}{p}}, & \t{ for }p>\frac{2d}{d-r}, \\
(\log N )^{\frac{d-r}{2d}} &  \t{ for }p=\frac{2d}{d-r},  \\
1, & \t { for }2\leq p< \frac{2d}{d-r}. 
\end{array}
\right.
$$
Moreover, this bound is sharp for all $p\geq 2$. 
\end{thm}

\begin{rem}[Limitation of our methods and future directions]
Firstly, we note that the deduction of Theorems~\ref{jointrestriction} and~\ref{gef} from Theorem~\ref{characterrestriction} via Schur's test relies in an essential way on the group structure of a maximal torus. At present it is unclear how to obtain analogous restriction bounds for sums of matrix coefficients or for general Laplace--Beltrami eigenfunctions on torus-generated conjugation-invariant submanifolds. 
Secondly, we restrict our attention to submanifolds of maximal flats and to torus-generated conjugation-invariant submanifolds, since our arguments crucially rely on the Weyl character formula, which directly encodes information about restrictions to tori. 
The restriction problem for general submanifolds would require different methods (see Table \ref{tab: submanifolds} for various submanifolds of $\mathrm{SU}(3)$). In particular, for subgroups, the problem could be studied using the asymptotic behavior of irreducible representations under restriction, as established by Heckman \cite{Hec82}. We leave the exploration of this intriguing direction to future work.

\end{rem}

\begin{table}[ht]
\centering
\renewcommand{\arraystretch}{1.4}
\begin{tabular}{|p{5cm}|p{6cm}|p{3cm}|}
\hline
\textbf{Submanifold} & \textbf{Description} &\textbf{Restriction Problem}\\
\hline
Maximal torus $T^2$ (dim $2$) & Diagonal unitary matrices with determinant $1$.  \newline Subgroup (and thus totally geodesic).& \it{Solved} \\
\hline
$\mathrm{SU}(2) $ (dim $3$) & Block-diagonal embedding 
$\Bigl(\begin{smallmatrix} \mathrm{SU}(2) & 0 \\ 0 & 1 \end{smallmatrix}\Bigr)$. \newline Diffeomorphic to $S^3$. \newline Subgroup (and thus totally geodesic). & \it{Unsolved}\\
\hline
$\mathrm{SO}(3) $ (dim $3$) & Real $3\times 3$ rotation matrices embedded as complex unitary matrices. \newline Diffeomorphic to $\mathbb{R}P^3$. \newline Subgroup and maximal totally geodesic.& \it{Unsolved} \\
\hline
$\mathrm{S}(\mathrm{U}(2)\times \mathrm{U}(1))$ (dim $4$) & Block-diagonal subgroup preserving a 2-plane. \newline Diffeomorphic to $(S^3 \times S^1)/\mathbb{Z}_2$. \newline Subgroup and maximal totally geodesic. & \it{Unsolved}\\
\hline
$\mathrm{SU}(3)/\mathrm{SO}(3)$ (dim $5$) & SU(3)-conjugates of all matrices in $\mathrm{SO}(3)$. \newline Maximal totally geodesic. & \it{Unsolved}\\
\hline
Partial flag manifold $\mathrm{SU}(3)/\mathrm{S}(\mathrm{U}(2)\times \mathrm{U}(1))$ (dim $4$) & Conjugacy class of any diagonal matrix with eigenvalue multiplicities $(2,1)$ (e.g., $\{1,-1,-1\}$). \newline Diffeomorphic to $\mathbb{C}P^2$. \newline Maximal totally geodesic.& \it{Solved only for characters} \\
\hline
Full flag manifold $\mathrm{SU}(3)/T^2$ (dim $6$) & Conjugacy class of any diagonal matrix with distinct eigenvalues. \newline Diffeomorphic to $F_{1,2}(\mathbb{C}^3)$.& \it{Solved only for characters}\\
\hline
\end{tabular}

\caption{Submanifolds of $\mathrm{SU}(3)$ and the restriction problem}\label{tab: submanifolds}
\end{table}

\subsection*{Overview of proof and organization of paper} 
The central result of this paper is Theorem \ref{characterrestriction}, the proof of which is given in Section \ref{MAIN}. As done in \cite{Zha23}, we incarnate the quantum-classical correspondence heuristics by first making a so-called ``barycentric-semiclassical'' subdivision of a fixed Weyl alcove, according to distance from the root hyperplanes. See Figure \ref{bssfigure} for the example of the group $\mathrm{SU}(3)$. Here the ``barycentric'' subdivision refers to distances as compared to a fixed scale, while the ``semiclassical'' subdivision refers to distances as compared to the ``Plank constant'' $1/N$. In particular, for the purpose of better exposition, the barycentric subdivision presented in this paper will be of a different form from that in \cite{Zha23}. In fact, Marshall in \cite{Mar16} applied a similar but finer dyadic subdivision of the Weyl alcove in a neighborhood of the origin. As it turns out, the finer dyadic subdivision is not needed for the purpose of this paper. 

For restriction bounds to a submanifold of the Weyl alcove, it suffices to consider each portion of the submanifold lying within the corresponding piece of the barycentric-semiclassical subdivision. See Figure \ref{curve} for the example of subdividing a curve in the alcove of $\mathrm{SU}(3)$. Based on the heuristics for the saturation of bounds in Theorem \ref{characterrestriction}, the curve $S$ in Figure \ref{curve}, chosen to be {\it tangential} to the left edge of the triangle, is expected to saturate the restriction bounds for characters, just as the edges of the triangle do.

Among the seven pieces $S_i$ ($i=1,\ldots,7$), we identify the likely dominant contributors to the restriction bound as follows. Piece $S_5$ may contribute the least, as it lies farthest from the edges. Comparing $S_3$ and $S_4$, we note that $S_4$ is both away from the bottom edge and not as close to the left edge, so it is less significant. A similar comparison between $S_4$ and $S_6$ further suggests that $S_4$ is not a primary contributor.

The remaining pieces are the main contenders, with $S_1$ and $S_7$ expected to be roughly equally significant, since both lie within a distance of $N^{-1}$ from the edges.

So who wins? A careful consideration suggests $S_2$. Compared with $S_3$, $S_2$ is closer to the bottom edge, a heuristic realized in {\bf Step 3} of the proof of Proposition \ref{keyprop}, which contains the main argument for Theorem \ref{characterrestriction}. Compared with $S_6$, $S_2$ is closer to the left edge, confirmed in {\bf Step 4}. Compared with $S_1$ (or $S_7$), although $S_1$ is closer to the bottom edge, it shrinks to a point as $N\to\infty$, whereas $S_2$ retains its substance and approaches the bottom edge in the limit; this is realized in {\bf Step 6}.

{\bf Step 1} of the proof of Proposition \ref{keyprop} involves selecting appropriate distances from the alcove walls to parameterize the submanifold, after taking a finite cover. In particular, both the finite cover and the parametrizations are made ``uniform'' for all translates of a fixed submanifold in the maximal torus, which is instrumental in the derivation of Theorem \ref{jointrestriction} from Theorem \ref{characterrestriction}.

The proof of Proposition \ref{keyprop} also involves two additional steps, namely {\bf Step 2} and {\bf Step 5}, which are closely related to the combinatorics of the arrangement of root hyperplanes. Near each vertex of the Weyl alcove, as mentioned earlier, we use distances from the alcove walls to parameterize the submanifold. In {\bf Step 2}, we order these distances to obtain another subdivision of (neighborhoods of) the alcove, allowing us to control the distances from all \emph{other} root hyperplanes containing the vertex, with the help of the combinatorial structure. These distances collectively determine the extent to which the characters are concentrated.  

Using the classification of root systems, we show in Section \ref{peelingrs} that, for each compact simple Lie group, there exist ``optimal'' ways to order the distances from the alcove walls near the origin. Compared to any other ordering near \emph{any} vertex, these orderings maximize the concentration of characters; this is illustrated in {\bf Step 5}.

Preparation work will be given in Section \ref{alcovereview} where we set up the notations about compact Lie groups, the Weyl alcoves and their facets, the key character formulas, and present the barycentric-semiclassical subdivision of the Weyl alcove. After proving Theorem \ref{characterrestriction} in Section \ref{peelingrs} and \ref{MAIN}, we demonstrate the sharpness of 
Theorem \ref{characterrestriction} in Section \ref{sharpness}. Theorem \ref{jointrestriction} and \ref{gef} are proved in Section \ref{application} and \ref{proofofgef} respectively. In Section \ref{invariantsub}, we take a closer look at torus-generated conjugation-invariant  submanifolds and in particular set up the Weyl integration formula adapted to them. 
In the last section, we prove Theorem \ref{singular} and \ref{singular2} along with their sharpness. 

\subsection*{Notation}
Throughout the paper, we write $a\lesssim b$ if $a \leq Cb$ for some
positive constant $C$, and write $a\asymp b$ if $a \lesssim b \lesssim a$. We write $a\ll b$ if there is a sufficiently small positive constant $c$ such that $a\leq cb$. 
For $1\leq p\leq\infty$, we use $p'$ to denote $1/(1-1/p)$.

\subsection*{Acknowledgments}
This project is partially supported by National Key R{\&}D Program of China (No. 2022YFA1006700). The author would like to thank Simon Marshall for helpful discussions and for making a visit to the University of Wisconsin--Madison possible. The author would also like to thank Jiaqi Hou and Xiaocheng Li for helpful discussions, and the anonymous referees for their valuable comments and suggestions.

\section{The Weyl alcove and the characters}\label{alcovereview}

We review basic structure and representation theory of compact Lie groups that can be found in standard texts such as \cite{Bou02}, \cite{Hel01} and \cite{Hel00}, and some machinery developed in \cite{Zha23} for the analysis of characters. 

\subsection{Structure of compact Lie groups and their alcoves}
Let $U$ be a compact connected simple Lie group with Lie algebra $\mathfrak{u}$. Let $\mathfrak{t}$ be a Cartan subalgebra, i.e. a maximal abelian subalgebra of $\mathfrak{u}$ and let $T$ be the corresponding analytic subgroup which is a maximal torus of $U$.  
Let $\mathfrak{t}^*$ denote the real dual space of $\mathfrak{t}$ and let $i$ denote the imaginary unit so that $i\mathfrak{t}^*$ is the space of linear forms on $\mathfrak{t}$ that take imaginary values. Let 
$\Sigma\subset i\mathfrak{t}^*$ be the root system of $(\mathfrak{u},\mathfrak{t})$. Fix a simple system $\{\alpha_1,\ldots,\alpha_r\}$ of $\Sigma$. Let $\alpha_0\in \Sigma$ be the corresponding lowest root. 
For $\alpha\in\Sigma$ and $n\in\Z$, define the root hyperplanes
$$\mathfrak{t}_{\alpha,n}:=\{H\in\mathfrak{t}:\ \alpha(H)/2\pi i+n=0\}.$$ 
These hyperplanes cut the ambient space $\mathfrak{t}$ into alcoves. 

For each $j=0,1,\ldots,r$, set the ``distance'' variables 
\begin{align}\label{distancev}
t_j(H):=\alpha_j(H)/2\pi i+\delta_{0j}
\end{align}
where $H\in\mathfrak{t}$. Here $\delta_{0j}$ equals 1 if $j=0$ and 0 otherwise. Let 
$$A:=\{H\in\mathfrak{t}:\ t_j(H)\geq 0,\ \forall j=0,\ldots,r\}$$ 
be the (closed) fundamental alcove. 
The walls of $A$ lie on the root hyperplanes $\mathfrak{t}_{\alpha_j,\delta_{0j}}$, $j=0,\ldots,r$. 
Under the exponential mapping $\exp: \mathfrak{t}\to T$, the alcove $A$ embeds in $T$, so we may also view $A$ as a subset of $T$. Let $W$ denote the finite Weyl group that acts on $T$, $\mathfrak{t}$ as well as on $\mathfrak{t}^*$.

The alcoves are simplices whose geometry may be described using the extended Dynkin diagram for $\Sigma$. Each $\alpha_j$ ($j=0,\ldots,r$) corresponds to a node in the extended Dynkin diagram (Figure \ref{dynkin}), and for each proper subset $J$ of $\{0,\ldots,r\}$, $\{\alpha_j, \ j\in J\}$ is a simple system for the rank-$|J|$ parabolic subsystem $\Sigma_J$ of $\Sigma$. The finite Dynkin diagram (or simply the Dynkin diagram) of $\Sigma_J$ can be obtained from the extended Dynkin diagram of $\Sigma$ by removing all the nodes not belonging to $J$. Associated to the simple system $\{\alpha_j, \ j\in J\}$ of $\Sigma_J$ is the positive system $\Sigma^+_J$ of $\Sigma_J$.

The facets of $A$ correspond to proper subsets of $\{0,\ldots,r\}$. For $J\subsetneqq\{0,\ldots,r\}$, 
\begin{align*}
A_J:&=\{H\in \mathfrak{t}:\ t_j(H)=0,\ \forall j\in J; \ t_j(H)>0,\ \forall j\notin J\} 
\end{align*}
is the corresponding $(r-|J|)$-dimensional facet.
We have 
$$A=\bigsqcup_{J\subsetneqq\{0,\ldots,r\}} A_J.$$ 

By the definition of $A_J$, if a root hyperplane $\mathfrak{t}_{\alpha,n}$ contains $A_J$, then $\alpha\in\Sigma_J$. 
For $J\subsetneqq\{0,\ldots,r\}$, let $W_J$ denote the finite Weyl group of $\Sigma_J$.

\begin{figure}
$\tilde{A}_1$:\,\dynkin[extended, labels={0,1}, edge length=1cm]A1\ 
$\tilde{A}_r$:\,\dynkin[extended, labels={0,1,2,r-1,r}, edge length=1cm]A{}\ 
$\tilde{B}_r$:\,\dynkin[extended, labels={0,1,2,3,r-2,r-1,r}, edge length=1cm]B{}\ 
$\tilde{C}_r$:\,\dynkin[extended, labels={0,1,2,r-2,r-1,r},  edge length=1cm]C{}\ 
$\tilde{D}_r$:\,\dynkin[extended, labels={0,1,2,3,r-3,,r-1,r}, labels*={,,,,,r-2,,}, edge length=1cm]D{}\ 

$\tilde{E}_6$:\,\dynkin[extended, labels={0,1,2,3,4,5,6},  edge length=1cm]E6\ 
$\tilde{E}_7$:\,\dynkin[extended, labels={0,1,2,3,4,5,6,7},  edge length=1cm]E7\ 
$\tilde{E}_8$:\,\dynkin[extended, labels={0,1,2,3,4,5,6,7,8},  edge length=1cm]E8\ 
$\tilde{F}_4$:\,\dynkin[extended , labels={0,1,2,3,4}, edge length=1cm]F4\ 
$\tilde{G}_2$:\,\dynkin[extended , labels={0,1,2}, edge length=1cm]G2

\caption{Extended Dynkin diagrams}\label{dynkin}
\end{figure}

\subsection{Barycentric-semiclassical subdivision} 
\label{sectionbarycentric}

From the semiclassical perspective, the characters of a compact Lie group should concentrate near focal points of the origin, which form the walls of alcoves of maximal tori. This motivates the semiclassical subdivision of the alcove according to how close the points are from each facet. 
Let $N$ be the growing parameter equal to the square root of the negative of the Laplace--Beltrami eigenvalue in question. For $J\subsetneqq\{0,\ldots,r\}$ let $A_J$ be the corresponding facet of $A$. We define a subset $P_J$ of $A$ that consists of points close to $A_J$ within a distance of $\lesssim N^{-1}$ but away from all the $A_K$ ($K\not\subset J$) by a distance of $\gtrsim N^{-1}$, or 
equivalently and more precisely, points that are $\leq N^{-1}$ close to the root hyperplanes $\mathfrak{t}_{\alpha_j,\delta_{0j}}$ for $j\in J$ while $>N^{-1}$ far from the other $\mathfrak{t}_{\alpha_j,\delta_{0j}}$ for $j\notin J$, namely,  
$$P_J:=\{H\in A:\ t_j(H)\leq N^{-1},\ \forall j\in J; \ t_j(H)> N^{-1},\ \forall j\notin J\}.$$ 
Then we have the so-called semiclassical subdivision (see Figure \ref{bssfigure})
\begin{align}
\label{semi}
A=\bigsqcup_{J\subsetneqq\{0,\ldots,r\} } P_J.
\end{align}

The fact that the points in $P_J$ are away from all the $A_K$ ($K\not\subset J$) by a distance of $\gtrsim N^{-1}$ is not going to be enough for our purpose. We would need to monitor for each point of $P_J$ a little more precisely how far it is from $A_K$ ($K\not\subset J$): is it close to $A_K$ within a small but fixed distance (independent of $N$) or away from $A_K$ by at least such a distance?
This can be done using a fattened version of the semiclassical subdivision. Namely, for a small but fixed positive constant $c$, set
$$\mathcal{N}_K:=\{H\in A:\ t_j(H)\leq c,\ \forall j\in K; \ t_j(H)> c,\ \forall j\notin K\}.$$ 
Just as $P_K$, $\mathcal{N}_K$ is a neighborhood of the barycenter of the facet $A_K$ ($K\subsetneqq \{0,\ldots,r\}$), and each $\mathcal{N}_K$ has the good property of staying close to the facet $A_K$ but away from all the $A_{K'}$ ($K'\not\subset K$) by at least a certain fixed distance. Thus we also have the so-called barycentric subdivision 
(see Figure \ref{bssfigure})
\begin{align}\label{bary}
A= \bigsqcup_{ K\subsetneqq \{0,\ldots,r\} } \mathcal{N}_K.
\end{align}

\begin{figure}
\scalebox{0.5}[0.5]
{

\begin{tikzpicture}

\node at (4,2.3094) [circle,fill,inner sep=1.5pt]{};

\node at (2,3.464) [circle,fill,inner sep=1.5pt]{};

\node at (6,3.464) [circle,fill,inner sep=1.5pt]{};

\node at (4,0) [circle,fill,inner sep=1.5pt]{};

\node at (0,0) [circle,fill,inner sep=1.5pt]{};

\node at (8,0) [circle,fill,inner sep=1.5pt]{};

\node at (4,6.928) [circle,fill,inner sep=1.5pt]{};

\draw (0, 0) -- (4, 6.928);
\draw (0,0) -- (8, 0);
\draw  (4, 6.928) -- (8, 0);

\draw (0.7217,1.25) -- (7.2783,1.25);
\draw (6.5566,0) -- (3.2783,5.6782);
\draw (4.7217, 5.6782) -- (1.4434,0);

\draw[<->] (3.5, 0) -- (3.5, 1.25); 
\node [right] at (3.5, 0.625) {\scalebox{2}{$ c$}};

\node [below] at (4, -0.5) {\scalebox{2}{(a)}};

\end{tikzpicture}

\hspace{1cm}
\begin{tikzpicture}

\draw (0, 0) -- (4, 6.928);
\draw (0,0) -- (8, 0);
\draw  (4, 6.928) -- (8, 0);

\draw (0.2887,0.5) -- (7.7113,0.5);
\draw (7.423,0) -- (3.7113,6.428);
\draw (4.2887, 6.428) -- (0.577,0);

\draw[<->] (4.5, 0) -- (4.5, 0.5); 
\node [right] at (4.5, 0.25) {\scalebox{2}{$  N^{-1}$}};

\node [below] at (4, -0.5) {\scalebox{2}{(b)}};

\end{tikzpicture}
\hspace{1cm}

\begin{tikzpicture}

\draw (0, 0) -- (4, 6.928);
\draw (0,0) -- (8, 0);
\draw  (4, 6.928) -- (8, 0);

\node at (4,2.3094) [circle,fill,inner sep=1.5pt]{};

\node at (2,3.464) [circle,fill,inner sep=1.5pt]{};

\node at (6,3.464) [circle,fill,inner sep=1.5pt]{};

\node at (4,0) [circle,fill,inner sep=1.5pt]{};

\node at (0,0) [circle,fill,inner sep=1.5pt]{};

\node at (8,0) [circle,fill,inner sep=1.5pt]{};

\node at (4,6.928) [circle,fill,inner sep=1.5pt]{};

\draw (0.2887,0.5) -- (7.7113,0.5);
\draw (7.423,0) -- (3.7113,6.428);
\draw (4.2887, 6.428) -- (0.577,0);

\draw[<->] (4.5, 0) -- (4.5, 0.5); 
\node [right] at (4.5, 0.25) {\scalebox{2}{$  N^{-1}$}};

\draw (0.7217,1.25) -- (7.2783,1.25);
\draw (6.5566,0) -- (3.2783,5.6782);
\draw (4.7217, 5.6782) -- (1.4434,0); 

\draw[-] (2.8, 3.1) -- (1.2, 3.1); 

\node [left] at (1.2, 3.1)
{\scalebox{2}{$P_{K,J}$}};

\draw[<->] (3.5, 0) -- (3.5, 1.25); 
\node [right] at (3.5, 0.625) {\scalebox{2}{$ c$}};

\node at (4,0) [circle,fill,inner sep=1.5pt]{};

\node at (0,0) [circle,fill,inner sep=1.5pt]{};

\node at (8,0) [circle,fill,inner sep=1.5pt]{};

\node at (4,6.928) [circle,fill,inner sep=1.5pt]{};

\node [below] at (4, -0.5) {\scalebox{2}{(c)}};

\draw [line width=2 pt] (2.1651,1.25) -- (1.2987,1.25) ;
\draw [line width=2 pt] (2.1651,1.25) -- (4,4.4284);
\draw [line width=2 pt] (4,4.4284) -- (3.56698, 5.17828); 
\draw [line width=2 pt] (3.56698, 5.17828) -- (1.2987,1.25);

\end{tikzpicture}

}
\caption{(a) Barycentric subdivision \ \  (b) Semiclassical subdivision \\ (c) Barycentric-semiclassical subdivision $A=\bigsqcup P_{K,J}$}
\label{bssfigure}
\end{figure}

\begin{figure}
\scalebox{0.5}[0.5]
{

\begin{tikzpicture}

\draw (0, 0) -- (4, 6.928);
\draw (0,0) -- (8, 0);
\draw  (4, 6.928) -- (8, 0);

\node at (4,2.3094) [circle,fill,inner sep=1.5pt]{};

\node at (2,3.464) [circle,fill,inner sep=1.5pt]{};

\node at (6,3.464) [circle,fill,inner sep=1.5pt]{};

\node at (4,0) [circle,fill,inner sep=1.5pt]{};

\node at (0,0) [circle,fill,inner sep=1.5pt]{};

\node at (8,0) [circle,fill,inner sep=1.5pt]{};

\node at (4,6.928) [circle,fill,inner sep=1.5pt]{};

\draw (0.2887,0.5) -- (7.7113,0.5);
\draw (7.423,0) -- (3.7113,6.428);
\draw (4.2887, 6.428) -- (0.577,0);

\draw[<->] (4.5, 0) -- (4.5, 0.5); 
\node [right] at (4.5, 0.25) {\scalebox{2}{$  N^{-1}$}};

\draw (0.7217,1.25) -- (7.2783,1.25);
\draw (6.5566,0) -- (3.2783,5.6782);
\draw (4.7217, 5.6782) -- (1.4434,0);

\draw[<->] (3.5, 0) -- (3.5, 1.25); 
\node [right] at (3.5, 0.625) {\scalebox{2}{$ c$}};

\node at (4,0) [circle,fill,inner sep=1.5pt]{};

\node at (0,0) [circle,fill,inner sep=1.5pt]{};

\node at (8,0) [circle,fill,inner sep=1.5pt]{};

\node at (4,6.928) [circle,fill,inner sep=1.5pt]{};

\draw [red] plot [smooth] coordinates {(0,0)
(1.2987-0.7217+0.2887,0.5) (2.1651,1.25) 
(4,3)
(4,4)
(3,4)
(0.85,1.25) 
(0.32,0.5)
(0,0)
};

\draw[-] (-0.45, 1.75/2) -- (0.55, 1.75/2); 
\node [left] at 
(-0.45, 1.75/2) {\scalebox{2}{$S_2$}};

\draw[-] (0.7, 2.6) -- (1.8, 2.6); 
\node [left] at 
(0.7, 2.6) {\scalebox{2}{$S_3$}};

\draw[-] (2.3, 5) -- (3.3, 4.2); 
\node [left] at 
(2.3, 5.2) {\scalebox{2}{$S_4$}};

\draw[-] (4.3, 2) -- (3.3, 2.2); 
\node [right] at 
(4.3, 2)  {\scalebox{2}{$S_5$}};

\draw[-] (1.5, 0.8) -- (2, -0.5); 
\node [below] at 
(2.2, -0.5)  {\scalebox{2}{$S_6$}};

\draw[-] (0.4, 0.2) -- (0.4, -0.5); 
\node [below] at 
(0.4, -0.5)  {\scalebox{2}{$S_7$}};

\draw[-] (0.14, 0.25) -- (-0.6, 0.25); 
\node [left] at 
(-0.6, 0.25) {\scalebox{2}{$S_1$}};

\end{tikzpicture}

}
\caption{A submanifold $S=\bigsqcup_{1\leq i\leq 7} S_i$ of the alcove}
\label{curve}
\end{figure}

\begin{rem}
This barycentric subdivision is chosen to be of a different form from that in \cite{Zha23} for the purpose of clearer exposition. It is also quite different from the standard operation that bears the same name in algebraic topology. We choose to abuse the terminology here; we will eventually also apply the standard barycentric subdivision of a cone as in \eqref{PIJcapNIsubset}, which will be built upon the above barycentric-semiclassical subdivision.  
\end{rem}

The following lemma is clear from the definition of $\mathcal{N}_K$. 

\begin{lem} \label{NKaway}
For $K\subsetneqq \{0,\ldots,r\}$, $\mathcal{N}_{K}$ stays a fixed distance away from all the root hyperplanes $\mathfrak{t}_{\alpha,n}$ with $\alpha\notin \Sigma_{K}$. 
\end{lem}

Now we combine the semiclassical and barycentric subdivision. The following lemma is also clear from definition. 

\begin{lem} \label{JsubsetK}
For $N$ large enough, we have 
$\mathcal{N}_K\cap P_J=\emptyset$ unless $J\subset K$. 
\end{lem}

For $J\subset K\subsetneqq \{0,\ldots,r\}$, define 
$$P_{K,J}=\mathcal{N}_K \cap P_J.$$
By \eqref{semi}, \eqref{bary}, and Lemma \ref{JsubsetK}, 
we have:

\begin{lem}[Barycentric-semiclassical subdivision] 
\label{BSS}
$$A=\bigsqcup_{ J\subset K\subsetneqq \{0,\ldots,r\} } P_{K,J}.$$
\end{lem}

\subsection{Behavior of characters across the alcove} \label{sectiondecomposition}
We give a formula of the character that would illuminate its behavior on each piece of the barycentric-semiclassical subdivision. 
Let $(\cdot,\cdot)$ denote the Killing form on $\mathfrak{u}$. Restricted on  
$\mathfrak{t}$, the Killing form gives the induced standard flat metric on $T$. The Killing form also extends to $\mathfrak{t}^*$ by duality and to $i\mathfrak{t}^*$ by linear extension. Let $|\cdot|$ denote the induced norm on $\mathfrak{t}$ and $\mathfrak{t}^*$ respectively, on which the Weyl group $W$ acts by isometries. 
The weight lattice reads 
$$\Lambda:=\left\{\mu\in i\mathfrak{t}^*:\ \frac{2(\mu,\alpha)}{(\alpha,\alpha)}\in\Z, \ \forall\alpha\in\Sigma\right\}.$$ The action of $W$ on $i\mathfrak{t}^*$ leaves $\Lambda$ invariant. 
 
Associated to any positive system $\Sigma^+$ of $\Sigma$ is the subset  
$$\Lambda^+:=\left\{\mu\in i\mathfrak{t}^*:\ \frac{2(\mu,\alpha)}{(\alpha,\alpha)}\in\mathbb{Z}_{\geq 1}, \ \forall\alpha\in\Sigma^+\right\}$$
of strictly dominant weights. We also set
$$\Lambda^\text{r}:=\left\{\mu\in \Lambda:\ \frac{2(\mu,\alpha)}{(\alpha,\alpha)}\neq 0, \ \forall\alpha\in\Sigma\right\}$$
to be the subset of regular weights. 
Let 
$$\rho:=\frac{1}{2}\sum_{\alpha\in\Sigma^+}\alpha$$
be the Weyl vector. 
Each $\mu\in\Lambda^+$ is associated with an irreducible representation of $U$ of highest weight $\mu-\rho$, and the associated character $\chi_\mu$ can be expressed by the Weyl formula 
\begin{align}\label{Weylcharacter}
\chi_\mu(\exp H)=\frac{\sum_{s\in W}\det s \ e^{(s\mu)(H)}}{\sum_{s\in W}\det s \ e^{(s\rho)(H)}}, \ \t{for }H\in\mathfrak{t}.
\end{align}
The Weyl denominator
$$\delta(H):=\sum_{s\in W}\det s \ e^{(s\rho)(H)}$$ can also be expressed as follows
\begin{align}\label{alphadelta}
\delta(H)=\prod_{\alpha\in\Sigma^+}\left(e^{\frac{\alpha(H)}{2}}-e^{-\frac{\alpha(H)}{2}}\right), \ \t{for }H\in\mathfrak{t}.
\end{align}
Evaluating the character at the identity, we get the Weyl formula for the dimension of the representation 
$$d_\mu=\frac{\prod_{\alpha\in\Sigma^+}(\alpha,\mu)}{\prod_{\alpha\in\Sigma^+}(\alpha,\rho)}.$$
Observe that the Weyl character formula \eqref{Weylcharacter} makes sense for any $\mu\in i\mathfrak{t}^*$; in particular, for $\mu\in\Lambda^+$ and $s\in W$, we have 
$$\chi_{s\mu}=\det s \cdot \chi_\mu.$$

We now factor the Weyl denominator. For any $J\subsetneqq \{0,1,\ldots,r\}$, set 
\begin{align}\label{deltaupperJ}\delta^J( H):=\prod_{\alpha\in\Sigma_J^+}\left(e^{\frac{\alpha(H)}{2}}-e^{-\frac{\alpha(H)}{2}}\right),
\end{align}
\begin{align}\label{deltalowerJ}
\delta_J( H):=\prod_{\alpha\in\Sigma^{+}\setminus\Sigma^+_J}\left(e^{\frac{\alpha(H)}{2}}-e^{-\frac{\alpha(H)}{2}}\right),
\end{align}
so that 
\begin{align}\label{factorWeyld}
\delta(H)=\delta_{J}(H)\cdot\delta^J(H)
\end{align}
for all $H\in\mathfrak{t}$. 
Here the positive system $\Sigma^+$ of $\Sigma$ is chosen as to contain $\Sigma_J^+$. We remark that later when we will be using Lemma \ref{BSS} the barycentric-semiclassical subdivision of the Weyl alcove, for each fixed pair $J\subset K\subsetneqq \{0,1,\ldots,r\}$, we will choose a positive system $\Sigma^+$ that contains $\Sigma^+_K$. For example, one can choose $\Sigma^{+}$ to be the set of all roots that are positive in the lexicographic ordering induced by the basis $\{\alpha_j,\ j\in I\}$ of $\Sigma$ for any $I$ containing $K$ ($I\subset \{0,1,\ldots,r\}$, $|I|=r$). Now for any $J\subset K\subsetneqq \{0,1,\ldots,r\}$, set 
\begin{align}\label{deltaJK}
\delta_{J}^K( H):=\prod_{\alpha\in \Sigma_K^+\setminus \Sigma^+_J}\left(e^{\frac{\alpha(H)}{2}}-e^{-\frac{\alpha(H)}{2}}\right).
\end{align}

We now study the behavior of $\chi_\mu$ near each facet of $A$. 
Consider the subspace 
\begin{align}\label{tJ}
\mathfrak{t}^J:=\bigoplus_{j\in J}\R H_{\alpha_j}
\end{align} 
of $\mathfrak{t}$, where $H_{\alpha_j}\in\mathfrak{t}$ is defined such that $(H_{\alpha_j},H):=\alpha_j(H)/2\pi i$ for all $H\in\mathfrak{t}$. Let $H^J$ denote the orthogonal projection of $H\in\mathfrak{t}$ on $\mathfrak{t}^J$ with respect to the Killing form. Let 
\begin{align}\label{HJperp}
H^{J\perp}:=H-H^J,
\end{align}
which lies in the orthogonal complement 
\begin{align}\label{tJperp}
\mathfrak{t}^{J\perp}:=\mathfrak{t}\ominus\mathfrak{t}^J
\end{align}
of $\mathfrak{t}^J$ in $\mathfrak{t}$. 
Dual to $\mathfrak{t}^J$, we also consider the root subspace $i\mathfrak{t}^{*J}=\bigoplus_{j\in J}\R\alpha_j$
of $i\mathfrak{t}^*$. For $\mu\in\Lambda$, let $\mu^J$ denote the orthogonal projection of $\mu$ on $i\mathfrak{t}^{*J}$. 
For $\gamma\in i\mathfrak{t}^{*J}$, let 
$$\chi^J_\gamma(\exp H^J)
:=\frac{\sum_{s_J\in W_J}\det s_J\ e^{(s_J\gamma)(H^J)}}{\delta^J(H^J)}$$
be the associated Weyl character. 
We then have the following key formula of characters.

\begin{lem}[The key character formula]\label{lemdecomposition}
For any $H\in\mathfrak{t}$ and $\mu\in i\mathfrak{t}^*$, we have 
\begin{align}\label{char}
\chi_\mu(\exp H)=\frac{1}{|W_J| \delta_J( H)}\sum_{s\in W}\det s\ e^{(s\mu)(H^{J\perp})}\chi^J_{(s\mu)^J}(\exp H^J).
\end{align}
\end{lem}
\begin{proof}
This is Lemma 4.2 of \cite{Zha23}. For completeness, we recall the proof. First, we rewrite the Weyl character formula as follows: for $H\in\mathfrak{t}$, we have 
\begin{align*}
\chi_\mu(\exp H)=\frac{\sum_{s\in W}\sum_{s_J\in W_J}\det (s_Js) \ e^{(s_J s\mu)(H)}}{|W_J|\prod_{\alpha\in \Sigma^+\setminus \Sigma^+_J}\left(e^{\frac{\alpha(H)}{2}}-e^{-\frac{\alpha(H)}{2}}\right)\prod_{\alpha\in\Sigma^+_J}\left(e^{\frac{\alpha(H)}{2}}-e^{-\frac{\alpha(H)}{2}}\right)}.
\end{align*}
Now write $H=H^J+H^{J\perp}$, we have for $s\in W$ and $s_J\in W_J$ that 
$$(s_J s\mu)(H)
=(s_J s\mu)(H^J)+(s_J s\mu)(H^{J\perp}).$$
But 
$$(s_J s\mu)(H^J)=(s_J(s\mu)^J)(H^J)$$
because
$(s\mu-(s\mu)^J)(H^J)=0$ by definition, while  
$$(s_J s\mu)(H^{J\perp})=(s\mu)(s_J^{-1} H^{J\perp})=(s\mu)(H^{J\perp})$$
because $s_J^{-1}$ as an element of $W_J$ fixes any point on $\mathfrak{t}^{J\perp}$. Also for $\alpha\in\Sigma_J$, 
$$\alpha(H)=\alpha(H^J)$$
because 
$\alpha(H^{J\perp})=0$. Combining the above identities we get 
$$\chi_\lambda=\frac{\sum_{s\in W}\det s\ e^{(s\mu)(H^{J\perp})}\sum_{s_J\in W_J}\det s_J e^{(s_J(s\mu)^J)(H^J)}}{|W_J|\prod_{\alpha\in \Sigma^+\setminus \Sigma^+_J}\left(e^{\frac{\alpha(H)}{2}}-e^{-\frac{\alpha(H)}{2}}\right)\prod_{\alpha\in\Sigma^+_J}\left(e^{\frac{\alpha(H^J)}{2}}-e^{-\frac{\alpha(H^J)}{2}}\right)}$$
which is \eqref{char}.
\end{proof}

We would need the following character bound.

\begin{lem}\label{charbound}
For $\mu\in\Lambda^+$, $s\in W$, $J\subsetneqq\{0,\ldots,r\}$, we have 
$$|\chi^J_{(s\mu)^J}(\exp H^J)|\lesssim |\mu|^{|\Sigma^+_J|},\ \t{for all }H\in \mathfrak{t}.$$
\end{lem}
\begin{proof}
As $s\mu$ is regular, 
$(s\mu)^J$ is regular with respect to the root subsystem $\Sigma_J$. This implies that there exists $s_J\in W_J$ such that $s_J(s\mu)^J$ lies in $\Lambda_J^+$, the set of strictly dominant weights with respect to the positive system
$\Sigma_J^+$ of $\Sigma_J$. Associated to $s_J(s\mu)^J$ is the character $\chi^J_{s_J(s\mu)^J}$ of the irreducible representation of the simply connected compact Lie group of root system $\Sigma_J$. Then we have 
$$\chi^J_{(s\mu)^J}=\det s_J \cdot \chi^J_{s_J(s\mu)^J}.$$
Now $|\chi^J_{s_J(s\mu)^J}|$ is bounded above by the associated dimension of representation 
$$d^J_{s_J(s\mu)^J}=\frac{\prod_{\alpha\in\Sigma_J^+}(\alpha, s_J(s\mu)^J)}{\prod_{\alpha\in\Sigma_J^+}(\alpha,\rho)}$$
which is clearly bounded by a constant multiple of $|\mu|^{|\Sigma_J^+|}$. The result follows. 
\end{proof}

\section{Peeling the root system}
\label{peelingrs}

In this section we demonstrate a  combinatorial rigidity in the arrangement of root hyperplanes. Let us order the distances from the Weyl alcove walls near a vertex so that 
$$t_{j_r}(H)\leq t_{j_{r-1}}(H)\leq\cdots\leq t_{j_1}(H)$$
where $j_1,\ldots,j_r$ is an $r$-permutation of $\{0,1,\ldots,r\}$. If we peel off those root hyperplanes the distance to which is comparable with the distance $t_{j_1}(H)$, then the remaining root hyperplanes correspond to the parabolic root subsystem $\Sigma_{\{j_1,\ldots,j_{r-1}\}}$. Then we peel off those root hyperplanes the distance to which is comparable with the distance $t_{j_2}(H)$, and so on. The numbers $k/p_k$ in Theorem \ref{characterrestriction} are going to be exactly the number of root hyperplanes one need to peel off from the original irreducible root system $\Sigma$ to get a rank-($r-k$) parabolic subsystem of $\Sigma$ which is of the largest cardinality among all rank-($r-k$) parabolic subsystems. We will demonstrate in the following lemma that actually these numbers $k/p_k$ can all be extracted from ways of ``peeling an irreducible root system most slowly'', and in every step of peeling to get a root subsystem of rank lower by one, one removes fewer roots than the previous step.

\begin{lem}\label{nk}
Let $\Sigma$ be an irreducible root system of rank $r$ and let $\{\alpha_j, \ j=0,1,\ldots,r\}$ be the extended simple system (containing the lowest root $\alpha_0$). Let 
$\mathcal{P}=(j_0, j_1, \ldots, j_{r})$ be any permutation of $\{0,1,\ldots,r\}$.  
For $i=0,1,\ldots,r$, let $I_i=I_i(\mathcal{P})=\{j_{i+1},\ldots,j_r\}$. In particular, $I_r=\emptyset$. Define 
\begin{align}\label{nk=}
n_i=n_i(\mathcal{P}):=|\Sigma^+|-|\Sigma^+_{I_{i}}|,\ i=0,1,\ldots,r
\end{align}
and the ``peeling numbers''
\begin{align}\label{qk=nk-}
q_i=q_i(\mathcal{P})=n_i-n_{i-1}=|\Sigma^+_{I_{i-1}}|-|\Sigma^+_{I_{i}}|,\ i=0,1,\ldots,r
\end{align}
where we have specified $n_{-1}=0$ and $\Sigma^+_{I_{-1}}:=\Sigma^+$.  
Then there exists a permutation $\mathcal{P}_0$ of $\{0,1,\ldots,r\}$, such that for any (other) permutation $\mathcal{P}$ of $\{0,1,\ldots,r\}$, the ``peeling inequality''
\begin{align}\label{peeling}
n_i(\mathcal{P}_0)\leq n_i(\mathcal{P})
\end{align}
holds for all $i=0,1,\ldots,r$. 
Moreover, letting $q_{i,0}=q_i(\mathcal{P}_0)$ ($i=0,1,\ldots,r$), then they are uniquely given in Table \ref{dismantling}, along with 
$$q_{0,0}\equiv 0.$$ 
In particular, an inspection of Table \ref{dismantling} reveals that  
\begin{align}\label{fewer}
q_{1,0}>q_{2,0}>\cdots>q_{r,0}=1.
\end{align}

\begin{center}
\begin{table}
{\renewcommand{\arraystretch}{1.2}
\begin{tabular}{|c|c|c|}
\hline  
$\Sigma$ &  $q_{1,0}$,  $q_{2,0}$, $\ldots$, $q_{r,0}$   
& $j_{1,0}$,  $j_{2,0}$, $\ldots$, $j_{r,0}$  
\\
\hline 
$A_r (r\geq 1)$ &  $r,r-1,\ldots,2,1$
& $1,2,\ldots,r$
\\

\hline 
$B_r (r\geq 2)$ & $2r-1,2r-3,\ldots,3,1$ 
&$1,2,\ldots,r$ \\

\hline 
$C_r (r\geq 3)$ &  $2r-1,2r-3,\ldots,3,1$
 &$1,2,\ldots,r$  \\

\hline 
$D_r (r\geq 4)$&  $2r-2,2r-4,\ldots,10,8,6,3,2,1$   
&   $1,2,\ldots,r-3,r-1,r-2,r$  \\

\hline 
$E_6$& $16,8,6,3,2,1$
& $1,6,2,3,4,5$ \\

\hline 
$E_7$ &  $27,16,8,6,3,2,1$
& $7,1,6,2,3,4,5$ \\

\hline 
$E_8$ &  $57,27,16,8,6,3,2,1$
& $8,7,1,6,2,3,4,5$ \\

\hline 
$F_4$&  $15,5,3,1$ 
& $1,4,2,3$ \\

\hline 
 $G_2$  & $5,1$
& $1,2$ \\

\hline

\end{tabular}
}
\caption{Peeling the root system}
\label{dismantling}

\end{table}
\end{center}

\end{lem}

\begin{proof}
Let $\mathcal{P}=(j_0,\ldots,j_r)$ be a permutation of $\{0,\ldots,r\}$. Consider removing the nodes in the extended Dynkin diagram as in Figure \ref{dynkin} one by one in the order of $j_0,j_{1},\ldots,j_r$. First, for the irreducible root systems a case-by-case inspection reveals that there indeed exist possibly multiple $\mathcal{P}_0$ which realizes the $r$-tuple $q_{1,0},\ldots,q_{r,0}$ in Table \ref{dismantling}; a choice of $\mathcal{P}_0=(j_{0,0},j_{1,0},\ldots,j_{r,0})$ is also given in Table \ref{dismantling} with $j_{0,0}\equiv 0$, using the labeling of the nodes in Figure \ref{dynkin}.  
The inequalities \eqref{peeling} require the desired $\mathcal{P}_0$ give the ``slowest way of peeling the irreducible root system'', so that at each step the remaining number of roots is no smaller than that from any other way. 

Define a ``complexity preorder'' $\leq$ on the set of (equivalence classes of) irreducible root systems using Table \ref{dismantling} as follows. For irreducible root systems $\Sigma$ and $\Sigma'$, suppose their optimal peeling numbers as appearing in the middle column of Table \ref{dismantling} are respectively $q_{1,0},q_{2,0},\ldots,q_{r,0}$ and $q'_{1,0},q'_{2,0},\ldots,q'_{r',0}$. We say $\Sigma\leq \Sigma'$ provided 
$$\sum_{i=0}^j q_{r-i,0}\leq\sum_{i=0}^j  q'_{r'-i,0}$$
for all $j=0,1,\ldots,\min\{r,r'\}-1$. And $\Sigma_1$ and $\Sigma_2$ are said to be equivalent with respect to this preorder, denoted $\Sigma_1\asymp \Sigma_2$, provided $\Sigma_1\leq \Sigma_2$ and $\Sigma_2\leq \Sigma_1$. It is clear from Table \ref{dismantling} that irreducible root systems of the same type ($A,B,C,D,E,F,G$) are equivalent to each other with respective to this complexity preorder, and the comparison between different types are given in Table \ref{complexity}. For example, $A_r\leq \Sigma$ for any irreducible root system $\Sigma$. For the comparison between $B,C$ and $E$, we have $E_6\leq B_r$ ($r\geq 2$), $E_6\leq C_r$ ($r\geq 3$), $E_s\leq B_r$ ($2\leq r\leq 6$, $s=6,7,8$), and $E_s\leq C_r$ ($3\leq r\leq 6$, $s=6,7,8$), but when both $r>6$ and $s>6$,  $E_s$ is not comparable with either $B_r$ or $C_r$.

\begin{center}
\begin{table}
{\renewcommand{\arraystretch}{1.2}
\begin{tabular}{|c|c|c|c|c|c|c|c|}
\hline  
  & $A$ & $B$ & $C$ & $D$ & $E$ & $F$ & $G$  \\
\hline 
$A$ &  $\asymp$    & $\leq$  & $\leq$  &  $\leq$ &   $\leq$& $\leq$  &  $\leq$  \\

\hline 
$B$ & $\geq$   & $\asymp$  & $\asymp$  &   $\geq$ &   & $\leq$  &  $\leq$  \\

\hline 
$C$ & $\geq$   & $\asymp$  & $\asymp$  &   $\geq$ &   & $\leq$    &  $\leq$  \\

\hline 
$D$ &  $\geq$   &  $\leq$  &  $\leq$  &  $\asymp$   &  $\leq$ & $\leq$  &  $\leq$  \\

\hline 
$E$ &  $\geq$  &   &   & $\geq$  &  $\asymp$   &  $\leq$  & $\leq$   \\

\hline 
$F$ &  $\geq$   & $\geq$   &  $\geq$  &  $\geq$   &   $\geq$ &  $\asymp$   &   $\leq$ \\
\hline 
$G$ &  $\geq$  &  $\geq$  &  $\geq$  &  $\geq$  &   $\geq$ &$\geq$  &$\asymp$     \\

\hline

\end{tabular}
}
\caption[The complexity preorder]{The complexity preorder}
\label{complexity}

\end{table}

\end{center}

A key observation is that after removing any node of the extended Dynkin diagram or the finite Dynkin diagram of any irreducible root system $\Sigma$, the remaining diagram is a union of connected components representing finite Dynkin diagrams of irreducible root subsystems $\Sigma'$ with the property $\Sigma'\leq \Sigma$. Using this observation, it is not hard to finish the proof by induction on the rank of $\Sigma$. We now give the full details anyways.

By transitivity of the complexity preorder, after removing the two nodes $j_0$, $j_1$ of the extended Dynkin diagram of $\Sigma$, the remaining diagram is also a union of connected components representing finite Dynkin diagrams of irreducible root subsystems $\Sigma'$ with the property $\Sigma'\leq \Sigma$; suppose there are two such connected components representing irreducible root subsystems $\Sigma_1$ and $\Sigma_2$ of $\Sigma$ with $\Sigma_i\leq \Sigma$ ($i=1,2$), and the argument will be entirely similar for only one or more than two connected components. For the remaining nodes we write
$$\{j_2,\ldots,j_r\}=\{j_{i_1},\ldots,j_{i_{m}}\}\bigsqcup \{j_{k_1},\ldots,j_{k_n}\}$$
where $2\leq i_1<i_2<\cdots<i_m\leq r$ and $2\leq k_1<k_2<\cdots<k_n\leq r$ such that 
$\{j_{i_1},\ldots,j_{i_{m}}\}$ and $\{j_{k_1},\ldots,j_{k_n}\}$ label respectively the nodes of the finite Dynkin diagrams of $\Sigma_1$ and $\Sigma_2$ as inherited from the initial labeling of the extended Dynkin diagram of $\Sigma$.  

Suppose that the optimal peeling numbers for $\Sigma$ and $\Sigma_1$ and $\Sigma_2$ are respectively $q_{1,0},\ldots,q_{r,0}$ and $q_{1,0}^{\Sigma_1},\ldots,q_{m,0}^{\Sigma_1}$ and $q_{1,0}^{\Sigma_2},\ldots,q_{n,0}^{\Sigma_2}$. Apply the induction hypothesis to $\Sigma_1$ and $\Sigma_2$, using $\Sigma_i\leq \Sigma$ ($i=1,2$), we have 
$$q_{j_{i_m}}+q_{j_{i_{m-1}}}+\cdots+q_{j_{i_{m-l+1}}}\leq q_{m,0}^{\Sigma_1}+q_{m-1,0}^{\Sigma_1}+\cdots+q_{m-l+1,0}^{\Sigma_1}\leq q_{r,0}+q_{r-1,0}+\cdots+q_{r-l+1,0}$$
for all $l=1,\ldots,m$, and 
$$q_{j_{k_n}}+q_{j_{k_{n-1}}}+\cdots+q_{j_{k_{n-l+1}}}\leq q_{n,0}^{\Sigma_2}+q_{n-1,0}^{\Sigma_2}+\cdots+q_{n-l+1,0}^{\Sigma_2}\leq q_{r,0}+q_{r-1,0}+\cdots+q_{r-l+1,0}$$
for all $l=1,\ldots,n$. At last, for each $u=1,\ldots,r-1$, we have 
$$\{j_r,j_{r-1},\ldots,j_{r-u+1}\}=\{j_{i_m},j_{i_{m-1}},\ldots,j_{i_{m-v+1}}\}\bigsqcup \{j_{k_n},j_{k_{n-1}},\ldots,j_{k_{n-u+v+1}}\}$$
for some $v=0,1,\ldots,u$, so using the above two inequalities we then have 
\begin{align*}
q_{j_{r}}+q_{j_{r-1}}+\cdots+q_{j_{r-u+1}}\leq  q_{r,0}+q_{r-1,0}+\cdots+q_{r-v+1,0} +q_{r,0}+q_{r-1,0}+\cdots+q_{r-u+v+1,0}.
\end{align*}
By \eqref{fewer}, this implies that for all $u=1,\ldots,r-1$, 
\begin{align*}
q_{j_{r}}+q_{j_{r-1}}+\cdots+q_{j_{r-u+1}}\leq  q_{r,0}+q_{r-1,0}+\cdots+q_{r-u+1,0}
\end{align*}
which is the same as the desired inequality \eqref{peeling}. The missing two cases are the trivial ones: $n_0(\mathcal{P}_0)=0\leq n_0(\mathcal{P})$ and $n_r(\mathcal{P}_0)=n_r(\mathcal{P})=|\Sigma^+|$. 
\end{proof}

\begin{rem}\label{nonorigin}
The above lemma may be compared with the Appendix in \cite{ST78}, where for each irreducible root system an $r$-tuple similar to but not always the same with that in Table \ref{dismantling} was given. For the arrangement of root hyperplanes containing the origin, those $r$-tuples in the Appendix of \cite{ST78} also provide lower bounds for the peeling inequality \eqref{peeling}, but they cannot be sharp as those tuples do not match ours in Table \ref{dismantling}. Another advantage of the above lemma is that it treats root hyperplane arrangements for an irreducible root system not just at the origin but at all vertices of the alcove uniformly, which contributes greatly to the simplicity of the proof of the theorems of this paper, as compared to \cite{Zha23} where we had to decompose the possibly reducible root systems associated to the non-origin vertices into irreducible pieces and ended up with a quite cumbersome argument. 
\end{rem}

\section{Proof of Theorem \ref{characterrestriction}}
\label{MAIN}

Let $S_0$ be a fixed compact smooth $k$-dimensional submanifold of the maximal torus $T$ ($k=0,1,\ldots,r$). Consider all the translates $S=xS_0=S_0x$ ($x\in T$). 
To evaluate the $L^p(S)$ (quasi-)norm of the characters $\chi$, as $\chi$ is conjugation-invariant and the conjugation action is transitive on the finite collection of alcoves in a maximal torus, it suffices to consider $L^p(S\cap A)$ where $A$ is the fundamental alcove. Using the barycentric-semiclassical subdivision as in Lemma \ref{BSS}, it suffices to estimate the $L^p$ norm of $\chi$ on each piece 
$$S_{K,J}:=S\cap P_{K,J}$$
where $J\subset K\subsetneqq \{0,\ldots,r\}$; see Figure \ref{curve} for the example of subdividing a curve in the alcove of $\mathrm{SU}(3)$. 

We need a good coordinate system for each $P_{K,J}$. For $K\subsetneqq \{0,\ldots,r\}$, recall the notations in \eqref{tJ}, \eqref{HJperp} and \eqref{tJperp}, we write 
$$\mathfrak{t}=\mathfrak{t}^{K}\oplus\mathfrak{t}^{K\perp}$$
so that for $H\in\mathfrak{t}$, 
\begin{align}\label{splitting}
H=H^{K}+H^{K\perp}
\end{align}
where $H^K\in\mathfrak{t}^{K}$ and $H^{K\perp}\in {t}^{K\perp}$.
By our construction, $P_{K,J}$ can be covered by a region in the form of  
\begin{align}\label{PKJcover}
P_J^K\times \mathcal{N}^{K\perp}=
\{H=H^{K}+H^{K\perp}: \ 
H^{K}\in P_{J}^K, \ 
H^{K\perp}\in\mathcal{N}^{K\perp}\}
\end{align}
where 
\begin{align}\label{PKJ}
P_{J}^K:=\{H^K\in\mathfrak{t}^K: \ 0\leq t_j(H^K)\leq N^{-1}, \ \forall j\in J; \ c\geq t_j(H^K)> N^{-1}, \ \forall j\in K\setminus J\}
\end{align}
for a small positive number $c$, and $\mathcal{N}^{K\perp}$ is a neighborhood in $\mathfrak{t}^{K\perp}$ (see Figure \ref{box}). 

\begin{figure}

    \begin{tikzpicture}
    \draw [line width=1 pt](2.1651/2,1.25/2) -- (1.2987/2,1.25/2) ;
    
    \draw  [line width=1 pt](3.56698/2, 5.17828/2) -- (1.2987/2,1.25/2);

(1.21642, 1.60707)

\draw  [line width=1 pt](4/2,4.4284/2) -- (3.56698/2, 5.17828/2); 
\draw  [line width=1 pt](2.1651/2,1.25/2) -- (4/2,4.4284/2);

(1.541275, 1.4196)

\draw [stealth-stealth](1.896904, 2.785554) -- (0.535936, 0.428586); 

\node [left] at (1.21642, 1.60707) {$\mathcal{N}^{K\perp}$};

\draw (1.896904, 2.785554) -- (2.221759, 2.598084); 

\draw [stealth-stealth](0.860791, 0.241116) -- (0.535936, 0.428586);

\node [below] at (0.6983635, 0.334851) 
{$P_J^K$}; 

\draw (1.45, 1.60707) -- (2, 1.60707); 
\node [right] at (2, 1.60707) 
{$P_{K,J}$}; 

\draw (0.860791, 0.241116) -- (2.221759, 2.598084); 

\end{tikzpicture}
\caption{$P_{K,J}\subset P_J^K\times\mathcal{N}^{K\perp}$}
\label{box}
\end{figure}

We use the character formula \eqref{char} to prove Theorem \ref{characterrestriction}. We now bound $L^p$ (quasi-)norms of the key term \eqref{deltalowerJ}
in the character formula.

\begin{prop}\label{keyprop}
Let $q_{i,0}$ ($i=0,1,\ldots,r$) be the optimal peeling numbers from Table \ref{dismantling} for the irreducible root system $\Sigma$. Let 
\begin{align}\label{pk=}
p_k=\frac{k}{q_{1,0}+q_{2,0}+\cdots+q_{k,0}}.
\end{align}
These are the critical exponents listed in Table \ref{criticalexponent}. 
Then 
\begin{align}\label{keybound}
\left\|\frac{1}{\delta_J}\right\|_{L^p(S_{K,J})}
\lesssim 
\left\{
\begin{array}{ll}
N^{|\Sigma^+|-|\Sigma_J^+|-\frac{k}{p}}, & \t{ for }p>p_k,\\
N^{|\Sigma^+|-|\Sigma_J^+|-\frac{k}{p_k}}(\log N)^{\frac{1}{p_k}}, & \t{ for }p=p_k,\\
N^{|\Sigma^+|-|\Sigma_J^+|-\frac{k}{p_k}}, & \t{ for }0<p<p_k.
\end{array}
\right.
\end{align}
Moreover, the above bounds hold with the implicit constant independent of $x\in T$, as in $S=xS_0=S_0x$. 
\end{prop}

\begin{proof}
\textbf{Step 1} (local parametrization of the submanifold $S$). Recall \eqref{deltaJK}. By Lemma \ref{NKaway}, for $H\in P_{K,J}\subset \mathcal{N}_{K}$, we have
\begin{align}\label{deltaIJK}
|\delta_{J}( H)|\asymp |\delta_{J}^K( H)|
\end{align}
which implies by the splitting \eqref{splitting} that 
\begin{align}\label{deltaJdeltaKJ}
|\delta_J( H)|\asymp |\delta_J^K( H^{K})|.
\end{align}

Fix any $K'\subset\{0,1,\ldots,r\}$ with $|K'|=r$ and $K\subset K'$. Then 
$$\{t_j:\ j\in K'\}$$
is a coordinate system for $\mathfrak{t}$, with 
$$\{t_j:\ j\in K\}$$
being the coordinate subsystem for $\mathfrak{t}^K$.  
By compactness of our $k$-dimensional submanifold $S_0$, we may pick a finite cover $\mathcal{F}$ of $S_0$ by its open subsets $V_0\in\mathcal{F}$ satisfying: for each $V_0$, there exists 
\begin{itemize}
\item 
a nonnegative integer $h$ such that $h\leq k$, $h\leq |K|$ and $k-h\leq |K'|-|K|=r-|K|$, 
\item a $(k-h)$-permutation $l_g$ ($g=h+1,\ldots,k$) of $K'\setminus K$,
\item 
a nonnegative integer $m$ such that 
$m\leq h$, $m\leq |K|-|J|$ and $h-m\leq |J|$, 
\item 
an $m$-permutation $l_g$ ($g=1,\ldots,m$) of $K\setminus J$, and 
\item 
an $(h-m)$-permutation $l_g$ ($g=m+1,\ldots,h$) of $J$, 
\end{itemize}
such that $V_0$ is parametrized by the variables $\{t_{l_g}:\ g=1,\ldots,k\}$, and that the submanifold volume measure $dH$ of $S_0$ restricted to $V_0$ satisfies 
\begin{align}\label{dH}
c \prod_{g=1}^k dt_{l_g}\leq dH\leq C \prod_{g=1}^k dt_{l_g}
\end{align}
for some positive constants $c,C$. 
Then automatically for all $x\in T$, as the $x$-translate of $S_0$, $S$ is covered by its open subsets $V=xV_0=V_0x$ ($V_0\in\mathcal{F}$); and more importantly, as a translate of $V_0$, $V$ has the same parametrization by the distance functions $\{t_{l_g}:\ g=1,\ldots,k\}$ as the above, for which the same estimate of the volume measure as \eqref{dH} holds.

Using the cover \eqref{PKJcover} of $P_{K,J}$, we have  
$$S_{K,J}:=S\cap P_{K,J}=\bigcup_V V\cap P_{K,J}\subset \bigcup_V V\cap (P_{J}^K\times\mathcal{N}^{K\perp}),$$
and so it suffices to estimate 
$\|1/\delta_J\|_{L^p(V\cap (P_{J}^K\times\mathcal{N}^{K\perp}))}$. By \eqref{PKJ}, \eqref{deltaJdeltaKJ} and \eqref{dH}, we have 
\begin{align}\label{firstbound}
\left\|\frac{1}{\delta_J}\right\|_{L^p(V\cap (P_{J}^K\times\mathcal{N}^{K\perp}))}
&\lesssim 
\left(\int_{
\substack{H^K\in P_J^K
\\
a\leq t_{l_g}\leq b, \ g=h+1,\ldots,k
}
}\left|\frac{1}{\delta_J^K( H^{K})}\right|^p
\ \prod_{g=1}^k\ dt_{l_g}\right)^{\frac{1}{p}}
\end{align}
where 
$$H^K=H^K(t_{l_1},\ldots,t_{l_k})$$
and $a,b$ are some universal constants independent of the submanifolds $S$, with the implicit constant independent of $x\in T$ as in $S=xS_0=S_0x$.

\textbf{Step 2} (barycentric subdivision of $P_{J}^K$ and the use of Fubini's theorem). We order the distances from the alcove walls. Recall \eqref{PKJ}, write 
\begin{align}\label{PIJcapNIsubset}
P_{J}^K=\bigcup_{\substack{(j_{r-|K|+1},j_{r-|K|+2},\ldots,j_{r-|J|})\t{ a}\\ \t{permutation of }K\setminus J}} \mathcal{R}_{j_{r-|K|+1},j_{r-|K|+2},\ldots,j_{r-|J|}},
\end{align}
where 
\begin{align*}
&\mathcal{R}_{j_{r-|K|+1},j_{r-|K|+2},\ldots,j_{r-|J|}}\\
&=\left\{H^K\in\mathfrak{t}^K: \ 0\leq t_j(H^K)\leq N^{-1},\ j\in J;\ N^{-1}<t_{j_{r-|J|}}(H^K)\leq\cdots\leq t_{j_{r-|K|+2}}(H^K)\leq t_{j_{r-|K|+1}}(H^K)\leq c\right\}.
\end{align*} 
So it suffices to estimate for each permutation $(j_{r-|K|+1},j_{r-|K|+2},\ldots,j_{r-|J|})$ of $K\setminus J$ the quantity 
\begin{align}\label{secondbound}
\left\|\frac{1}{\delta_J}\right\|_{L^p(V\cap (\mathcal{R}_{j_{r-|K|+1},j_{r-|K|+2},\ldots,j_{r-|J|}}\times\mathcal{N}^{K\perp}))}
&\lesssim 
\left(\int_{
\substack{H^K\in \mathcal{R}_{j_{r-|K|+1},j_{r-|K|+2},\ldots,j_{r-|J|}}
\\
a\leq t_{l_g}\leq b, \ g=h+1,\ldots,k
}
}\left|\frac{1}{\delta_J^K( H^{K})}\right|^p
\ \prod_{g=1}^k\ dt_{l_g}\right)^{\frac{1}{p}}.
\end{align}

Now we involve the peeling numbers. For each $\alpha\in\Sigma^+_K\setminus \Sigma^+_J$, write 
$$\alpha=\sum_{j\in K} n_j^\alpha \alpha_j$$ 
where the coefficients $n_j^\alpha$'s ($j\in K$) are all nonnegative integers and there is at least one $j\in K\setminus J$ for which $n_j^\alpha$ is positive. For $H^K\in P_J^K$, taking the positive constant $c$ in the definition \eqref{PKJ} of $P_J^K$ small enough, we have 
\begin{align}\label{enj}
\left|e^{\frac{\alpha(H^K)}{2}}-e^{-\frac{\alpha(H^K)}{2}}\right|
\asymp \sum_{j\in K} n_j^\alpha t_j(H^K).
\end{align}
For each permutation $(j_{r-|K|+1},j_{r-|K|+2},\ldots,j_{r-|J|})$ of $K\setminus J$, let $q_i$ ($i=r-|K|+1,r-|K|+2,\ldots,r-|J|$) be the associated peeling numbers as defined in \eqref{qk=nk-}, i.e., 
\begin{align}\label{defqi}
q_i=|\Sigma^+_{I_{i-1}}|-|\Sigma^+_{I_{i}}|
\end{align}
where $I_i=\{j_{i+1},j_{i+2},\ldots,j_{r-|J|}\}\cup J $. Then
$q_i$ is exactly the number of $\alpha$ in $\Sigma^+_K\setminus \Sigma^+_J$ such that the coefficients $n_j^\alpha$'s are all zero for $j=j_k$ with $r-|K|+1\leq k<i$ while the coefficient $n_{j_i}^\alpha$ is positive; for such an $\alpha$ and for $H^K\in \mathcal{R}_{j_{r-|K|+1},j_{r-|K|+2},\ldots,j_{r-|J|}}$, by \eqref{enj}, we then have 
\begin{align*}
\left|e^{\frac{\alpha(H^K)}{2}}-e^{-\frac{\alpha(H^K)}{2}}\right|
\asymp t_{j_i}(H^K).
\end{align*}
Hence for $H^K\in \mathcal{R}_{j_{r-|K|+1},j_{r-|K|+2},\ldots,j_{r-|J|}}$, it holds 
\begin{align}\label{deltaKJasymp}
|\delta_J^K( H^K)|=\prod_{\alpha\in \Sigma^+_K\setminus \Sigma^+_J} \left|e^{\frac{\alpha(H^K)}{2}}-e^{-\frac{\alpha(H^K)}{2}}\right|\asymp t_{j_{r-|K|+1}}^{q_{r-|K|+1}}(H^K)t_{j_{r-|K|+2}}^{q_{r-|K|+2}}(H^K)\cdots t_{j_{r-|J|}}^{q_{r-|J|}}(H^K).
\end{align}

Recall from {\bf Step 1} that $l_g$ ($g=1,\ldots,m$) is an $m$-permutation of $K\setminus J$. Then we can rewrite 
$$l_g=j_{n_{k-h+g}}, \ g=1,\ldots,m$$ 
for some 
$$r-|K|+1\leq n_{k-h+1}<\cdots<n_{k-h+m}\leq r-|J|.$$
Using \eqref{deltaKJasymp}, we estimate
\begin{align*}
&\left\|\frac{1}{\delta_J}\right\|_{L^p(V\cap (\mathcal{R}_{j_{r-|K|+1},j_{r-|K|+2},\ldots,j_{r-|J|}}\times\mathcal{N}^{K\perp}))}\\
&\lesssim
\bigg(\int_{\substack{
N^{-1}<t_{j_{r-|J|}}(H^K)\leq\cdots\leq t_{j_{r-|K|+1}}(H^K)\leq c
\\
0\leq t_{l_g}\leq N^{-1},\ g=m+1,\ldots,h
\\ 
a\leq t_{l_g}\leq b, \ g=h+1,\ldots,k}
} t_{j_{r-|K|+1}}^{-q_{r-|K|+1}p}(H^K)\cdots t_{j_{r-|J|}}^{-q_{r-|J|}p}(H^K)\\
&\ \ \ \ \ \  \cdot 
\prod_{g=1}^m\ dt_{j_{n_{k-h+g}}}
\prod_{g=m+1}^k\ dt_{l_g}
\bigg)^{\frac{1}{p}}\\
&\lesssim 
N^{q_{n_{k-h+m}+1}+\cdots+q_{r-|J|}}\cdot 
\bigg(\int_{\substack{
N^{-1}<t_{j_{n_{k-h+m}}}\leq\cdots\leq t_{j_{n_{k-h+1}}}\leq c
\\
0\leq t_{l_g}\leq N^{-1},\ g=m+1,\ldots,h
\\ 
a\leq t_{l_g}\leq b, \ g=h+1,\ldots,k}
} \\
&\ \ \ \ \ \  t_{j_{n_{k-h+1}}}^{-(q_{r-|K|+1}+\cdots+q_{n_{k-h+1}})p}\cdots t_{j_{n_{k-h+m}}}^{-(q_{n_{k-h+m-1}+1}+\cdots+q_{n_{k-h+m}})p}
\prod_{g=1}^m\ dt_{j_{n_{k-h+g}}}
\prod_{g=m+1}^k\ dt_{l_g}
\bigg)^{\frac{1}{p}}\\
&\lesssim 
N^{-\frac{h-m}{p}+q_{n_{k-h+m}+1}+\cdots+q_{r-|J|}} \bigg(\int_{
N^{-1}<t_{j_{n_{k-h+m}}}\leq\cdots\leq t_{j_{n_{k-h+1}}}\leq c
} \\
&\ \ \ \ \ \  
t_{j_{n_{k-h+1}}}^{-(q_{r-|K|+1}+\cdots+q_{n_{k-h+1}})p}\cdots t_{j_{n_{k-h+m}}}^{-(q_{n_{k-h+m-1}+1}+\cdots+q_{n_{k-h+m}})p}
\prod_{g=1}^m\ dt_{j_{n_{k-h+g}}}
\bigg)^{\frac{1}{p}}.
\end{align*}
In particular, the second inequality above holds with the implicit constant independent of the function $H^K=H^K(t_1,\ldots,t_k)$! We have arrived at an estimate with the implicit constant independent of $x\in T$ as in $S=xS_0=S_0x$. 

\textbf{Step 3} (the first reverse use of Fubini's theorem). 
Now we add $k-h$ variables back to the integration! 
Pick any permutation $j_{0},\ldots,j_{r-|K|}$ of $\{0,1,\ldots,r\}\setminus K$. As $k-h\leq r-|K|$, we can pick any subsequence $j_{n_{1}},\ldots,j_{n_{k-h}}$ of $j_{1},\ldots,j_{r-|K|}$ such that 
$$1\leq n_{1}<\cdots<n_{k-h}\leq r-|K|.$$
Associated to the tuple $j_{0},\ldots,j_{r-|J|}$, the peeling numbers $q_i$'s can now be defined for all $i=0,\ldots,r-|J|$ as in \eqref{defqi}. We also put $n_0=0$ for convenience. 

We estimate 
\begin{align*}
&\left\|\frac{1}{\delta_J}\right\|_{L^p(V\cap (\mathcal{R}_{j_{r-|K|+1},j_{r-|K|+2},\ldots,j_{r-|J|}}\times\mathcal{N}^{K\perp}))}\\
&\lesssim 
N^{-\frac{h-m}{p}+q_{n_{k-h+m}+1}+\cdots+q_{r-|J|}}
\bigg(\int_{   
N^{-1}<t_{j_{n_{k-h+m}}}\leq\cdots\leq t_{j_{n_{k-h+1}}}\leq c
}\\
&\ \ \ \ \ \ t_{j_{n_{k-h+1}}}^{-(q_{n_{k-h}+1}+\cdots+q_{n_{k-h+1}})p}\cdots t_{j_{n_{k-h+m}}}^{-(q_{n_{k-h+m-1}+1}+\cdots+q_{n_{k-h+m}})p}
\prod_{g=1}^m\  dt_{j_{n_{k-h+g}}}
\bigg)^{\frac{1}{p}}
\\
&\lesssim 
N^{-\frac{h-m}{p}+q_{n_{k-h+m}+1}+\cdots+q_{r-|J|}}
\bigg(\int_{
\substack{   
N^{-1}<t_{j_{n_{k-h+m}}}\leq\cdots\leq t_{j_{n_{k-h+1}}}\leq c
\\
c\leq t_{j_{n_{k-h}}}\leq\cdots\leq t_{j_{n_1}} \leq 2c
}
}\\
&\ \ \ \ \ \ t_{j_{n_{k-h+1}}}^{-(q_{n_{k-h}+1}+\cdots+q_{k-h+1})p}\cdots t_{j_{n_{k-h+m}}}^{-(q_{n_{k-h+m-1}+1}+\cdots+q_{n_{k-h+m}})p}
\ dt_{j_{n_1}}\cdots\ dt_{j_{n_{k-h+m}}}
\bigg)^{\frac{1}{p}}
\\
&\lesssim 
N^{-\frac{h-m}{p}+q_{n_{k-h+m}+1}+\cdots+q_{r-|J|}}
\bigg(\int_{
\substack{   
N^{-1}<t_{j_{n_{k-h+m}}}\leq\cdots\leq t_{j_{n_{k-h+1}}}\leq c
\\
c\leq t_{j_{n_{k-h}}}\leq\cdots\leq t_{j_{n_1}} \leq 2c
}
} \\
&\ \ \ \ \ \ t_{j_{n_{1}}}^{-(q_1+\cdots+q_{n_1})p}\cdots t_{j_{n_{k-h+m}}}^{-(q_{n_{k-h+m-1}+1}+\cdots+q_{n_{k-h+m}})p}
\ dt_{j_{n_1}}\cdots\ dt_{j_{n_{k-h+m}}}
\bigg)^{\frac{1}{p}}
\\
&\lesssim 
N^{-\frac{h-m}{p}+q_{n_{k-h+m}+1}+\cdots+q_{r-|J|}}
\bigg(\int_{  
N^{-1}<t_{j_{n_{k-h+m}}}\leq\cdots\leq t_{j_{n_{1}}}\leq 2c
}\\
&\ \ \ \ \ \ t_{j_{n_{1}}}^{-(q_1+\cdots+q_{n_1})p}\cdots t_{j_{n_{k-h+m}}}^{-(q_{n_{k-h+m-1}+1}+\cdots+q_{n_{k-h+m}})p}
\ dt_{j_{n_1}}\cdots\ dt_{j_{n_{k-h+m}}}
\bigg)^{\frac{1}{p}}
\\
&\lesssim 
N^{-\frac{h-m}{p}+q_{n_{k-h+m}+1}+\cdots+q_{r-|J|}}
\bigg(\int_{  
N^{-1}<t_{j_{n_{k-h+m}}}\leq\cdots\leq t_{j_{n_{1}}}\leq 2c
}\\
&\ \ \ \ \ \ 
t_{j_{n_{1}}}^{-(q_{0}+q_1+\cdots+q_{n_1})p}\cdots t_{j_{n_{k-h+m}}}^{-(q_{n_{k-h+m-1}+1}+\cdots+q_{n_{k-h+m}})p}
\ dt_{j_{n_1}}\cdots\ dt_{j_{n_{k-h+m}}}
\bigg)^{\frac{1}{p}}.
\end{align*}

{\bf Step 4} (``upgrading'' the parameters).  
Set 
$$s_g:=t_{j_{n_g}}$$ 
for $g=1,\ldots,k-h+m$. Using crucially the property 
\begin{align}\label{crucial}
N^{-1}< s_{k-h+m}\leq\cdots\leq s_1,
\end{align}
we can further estimate 
\begin{align*}
&\left\|\frac{1}{\delta_J}\right\|_{L^p(V\cap (\mathcal{R}_{j_{r-|K|+1},j_{r-|K|+2},\ldots,j_{r-|J|}}\times\mathcal{N}^{K\perp}))}\\
&\lesssim 
N^{-\frac{h-m}{p}+q_{k-h+m+1}+\cdots+q_{r-|J|}}\\
&\ \ \ \cdot 
\bigg(\int_{  
N^{-1}<s_{k-h+m}\leq\cdots\leq s_{1}\leq 2c
} 
s_{1}^{-(q_{0}+q_1)p}
s_2^{-q_2p}\cdots s_{k-h+m}^{-q_{k-h+m}p}
\ ds_{1}\cdots\ ds_{k-h+m}
\bigg)^{\frac{1}{p}}.
\end{align*}

{\bf Step 5} (using the peeling inequality). Now we can use the peeling inequality \eqref{peeling}, which gives 
$$q_0+q_1+\cdots+q_g\geq q_{0,0}+q_{1,0}+\cdots+q_{g,0}=q_{1,0}+\cdots+q_{g,0}$$
for all $g=1,\ldots,k-h+m$. Again exploiting \eqref{crucial}, the above peeling inequality implies  
\begin{align*}
&N^{|\Sigma_J^+|}\cdot \left\|\frac{1}{\delta_J}\right\|_{L^p(V\cap (\mathcal{R}_{j_{r-|K|+1},j_{r-|K|+2},\ldots,j_{r-|J|}}\times\mathcal{N}^{K\perp}))}\\
&\lesssim 
N^{-\frac{h-m}{p}+q_{k-h+m+1}+\cdots+q_{r-|J|}+|\Sigma_J^+|}\\
&\ \ \ \cdot \bigg(
\int_{N^{-1}< s_{k-h+m}\leq\cdots\leq s_1\leq 2c} 
s_1^{-(q_0+q_1)p}s_2^{-q_2p}\cdots s_{k-h+m}^{-q_{k-h+m}p} 
\ ds_{1}\cdots\ ds_{k-h+m}
\bigg)^{\frac{1}{p}}\\
&\lesssim 
N^{-\frac{h-m}{p}+q_{k-h+m+1,0}+\cdots+q_{r,0}}\\
&\ \ \ \cdot \bigg(
\int_{N^{-1}< s_{k-h+m}\leq\cdots\leq s_1\leq 2c} 
s_1^{-q_{1,0}p}s_2^{-q_{2,0}p}\cdots s_{k-h+m}^{-q_{k-h+m,0}p} 
\ ds_{1}\cdots\ ds_{k-h+m}
\bigg)^{\frac{1}{p}}.
\end{align*}
Here we also used the equality 
$$q_0+q_1+\cdots+q_{r-|J|}+|\Sigma^+_J|=q_{1,0}+q_{2,0}+\cdots+q_{r,0}=|\Sigma^+|.$$

\textbf{Step 6} (the second reverse use of Fubini's theorem). 
For the last step, we add another $h-m$ variables $s_{k-h+m+1},\ldots,s_k$ back to the integration! Then we can further estimate 
\begin{align*}
&N^{|\Sigma_J^+|}\cdot \left\|\frac{1}{\delta_J}\right\|_{L^p(V\cap (\mathcal{R}_{j_{r-|K|+1},j_{r-|K|+2},\ldots,j_{r-|J|}}\times\mathcal{N}^{K\perp}))}\\
&\lesssim 
N^{q_{k+1,0}+\cdots+q_{r,0}} \bigg(
\int_{\substack{ 0.5N^{-1}\leq s_{k}\leq s_{k-1}\leq\cdots\leq s_{k-h+m+1}\leq N^{-1}\\ N^{-1}< s_{k-h+m}\leq\cdots\leq s_1\leq 2c }} 
s_1^{-q_{1,0}p}\cdots s_{k}^{-q_{k,0}p} 
\ ds_{1}\cdots\ ds_{k}
\bigg)^{\frac{1}{p}}\\
&\lesssim N^{q_{k+1,0}+\cdots+q_{r,0}} \bigg(
\int_{0.5 N^{-1}< s_k\leq \cdots\leq s_1\leq 2c} 
s_1^{-q_{1,0}p}\cdots  s_k^{-q_{k,0}p}
\ ds_{1}\cdots ds_k
\bigg)^{\frac{1}{p}}.
\end{align*}
So we end up with an integral that could be complicated to evaluate if the parameters $q_{1,0},\ldots,q_{k,0}$ were arbitrary. Fortunately in our case, we have \eqref{fewer}, and the proposition now follows by an easy exercise in multiple integrals which we record below as a lemma. 
\end{proof}

\begin{lem}\label{integration}
Let $a_1>a_2>\cdots>a_k$ and $c$ be positive numbers given as constants. Let $N$ be a parameter that takes large positive values. Let $A=a_1+\cdots+a_k$, and $p_0=k/A$. Then 
\begin{align*}
\left(
\int_{N^{-1}< s_k\leq s_{k-1}\leq \cdots\leq s_1\leq c} 
s_1^{-a_1p}\cdots  s_k^{-a_kp}
\ ds_{1}\cdots ds_k 
\right)^{\frac{1}{p}}
\asymp \left\{
\begin{array}{lll}
 N^{A-\frac{k}{p}}, & \t{ for }p>p_0, 
\\
(\log N )^{\frac{1}{p_0}}, &\t{ for }p=p_0, \\
1, & \t { for }0<p< p_0. 
\end{array}
\right.
\end{align*}
\end{lem}

Now we can finish the proof of Theorem \ref{characterrestriction}. 
For $\mu\in\Lambda^+$, the Laplace--Beltrami eigenvalue of $\chi_\mu$ equals $-|\mu|^2+|\rho|^2=-N^2$ (see Chapter 5 of \cite{Hel00}). As discussed at the beginning of this section, it suffices to get the correct $L^p$ bound of $\chi_\lambda$ on each $S_{K,J}$ ($J\subset K\subsetneqq \{0,\ldots,r\}$).  
Using \eqref{char}, we have
$$\left|\chi_\mu(\exp H)\right|
\leq\frac{1}{|W_J|\cdot  |\delta_J( H)|}\sum_{s\in W}\left|\chi^J_{(s\mu)^J}(\exp H^J)\right|.$$
Now Lemma \ref{charbound} gives that for $H\in S_{K,J}\subset P_J$, 
$$\left|\chi_\mu(\exp H)\right|
\lesssim N^{|\Sigma^+_J|}\cdot \frac{1}{|\delta_J( H)|}.$$
The desired bound follows from Proposition \ref{keyprop}. Sharpness of the bound will be demonstrated in the next section.

\section{Sharpness of Theorem \ref{characterrestriction}}
\label{sharpness}
In this section we establish that the character bounds in Theorem \ref{characterrestriction} are sharp for all $p>0$. It then follows that the bounds in Theorem \ref{jointrestriction} are sharp for all $p\geq 2$, with the sole exception noted therein. 
As discussed in the Introduction, among the $k$-dimensional submanifolds of the maximal torus $T$, we pick the $k$-dimensional facet $A_J$ of the alcove $A$ 
that lies on the largest number of root hyperplanes among all $k$-dimensional facets to saturate the bound, and we can construct $A_J$ by peeling the root system in the slowest manner. Using Lemma \ref{nk}, for any permutation $\mathcal{P}_0=(j_0,j_1,\ldots,j_r)$ of $\{0,1,\ldots,r\}$ that realizes $q_i(\mathcal{P}_0)=q_{i,0}$ for all $i=0,1,\ldots,r$, set $J:=\{j_{k+1},\ldots,j_{r}\}$. 

\begin{thm}
Let $\mu=N\rho\in\Lambda^+$, where $N$ is any (large) natural number. Then $\|\chi_\mu\|_{L^p(A_J)}$ saturate the bounds in Theorem \ref{characterrestriction} for all $p>0$. 
\end{thm}
\begin{proof}
Define 
\begin{align}\label{submanifold}
S_N:=\{H\in A_J:\ N^{-1}< t_{j_{k}}(H)\leq\cdots\leq t_{j_{1}}(H)\leq c\}.
\end{align} 
Recall that $q_{0,0}=0$, hence $\alpha_{j_1},\alpha_{j_2},\ldots,\alpha_{j_r}$ is a simple system for $\Sigma$ which induces the positive system $\Sigma^+$. 
By a similar reasoning that arrives at \eqref{deltaKJasymp}, we can choose $c$ small enough such that 
\begin{align}\label{deltaJIs}
|\delta _{J}( H)|=\prod_{\alpha\in\Sigma^+\setminus \Sigma_J^+}\left|e^{\frac{\alpha(H)}{2}}-e^{-\frac{\alpha(H)}{2}}\right|\asymp  t_{j_{1}}^{q_{1,0}}(H) \cdots t_{j_{k}}^{q_{k,0}}(H)
\end{align} 
uniformly for $H\in S_N$. 
We will show that $\|\chi_\mu\|_{L^p(S_N)}$ already saturates the bounds in Theorem \ref{characterrestriction} for all $p>0$. By the Weyl character formula, we have  
\begin{align*}
\chi_\mu(\exp H)&=\frac{\sum_{s\in W}\det s \ e^{(s(N\rho))(H)}}{\sum_{s\in W}\det s \ e^{(s\rho)(H)}}\\
&=\frac{\sum_{s\in W}\det s \ e^{(s\rho)(NH)}}{\sum_{s\in W}\det s \ e^{(s\rho)(H)}}\\
&=\frac{\prod_{\alpha\in\Sigma^+}\left(e^{\frac{\alpha(NH)}{2}}-e^{-\frac{\alpha(NH)}{2}}\right)}{\prod_{\alpha\in\Sigma^+}\left(e^{\frac{\alpha(H)}{2}}-e^{-\frac{\alpha(H)}{2}}\right)}.
\end{align*}
For $H\in A_J$, it holds
\begin{align}\label{CSharp}
|\chi_\mu(\exp H)|=\left(\prod_{\alpha\in\Sigma^+_J}\lim_{H_\varepsilon\to 0} \frac{e^{\frac{\alpha(NH_\varepsilon)}{2}}-e^{-\frac{\alpha(NH_\varepsilon)}{2}}}{e^{\frac{\alpha(H_\varepsilon)}{2}}-e^{-\frac{\alpha(H_\varepsilon)}{2}}}\right)
 \frac{|\delta_J( NH)|}{|\delta_J( H)|}
 = N^{|\Sigma_J^+|}\frac{|\delta_J( NH)|}{|\delta_J( H)|}.
\end{align}

As the volume measure $dH$ on $S_N$ satisfies 
$$dH\asymp \ dt_{j_1}\cdots\ dt_{j_k},$$
by \eqref{deltaJIs}, we have 
\begin{align*}
\left\| \frac{1}{\delta_J}\right\|_{L^p(S_N)}
&\asymp  
\left(\int_{N^{-1}<t_{j_{k}}\leq\cdots\leq t_{j_1}\leq c} t_{j_1}^{-q_{1,0}p}\cdots t_{j_{k}}^{-q_{k,0}p}\ dt_{j_{1}}\cdots dt_{j_k}\right)^{\frac{1}{p}}
\end{align*}
which by \eqref{fewer} and Lemma \ref{integration} saturates the bound on the right side of \eqref{keybound}.

\begin{figure}

\begin{tikzpicture}

\draw (0, 0) -- (4/2, 6.928/2);
\draw (0,0) -- (8/2, 0);
\draw  (4/2, 6.928/2) -- (8/2, 0);

\draw   (2/2, 3.464/2) --(4/2,0) ; 

\draw (2/2, 3.464/2) -- (6/2, 3.464/2);

\draw (4/2,0) -- (6/2, 3.464/2);

\draw (5/2, 5.196/2) -- (2/2, 0);

\draw (3/2,5.196/2) -- (6/2, 0);

\draw (1/2,1.732/2) -- (7/2, 1.732/2);

\draw (3/2,5.196/2) --  (5/2,5.196/2) ;

\draw (2/2, 0)--(1/2,1.732/2);

\draw (6/2, 0)--(7/2, 1.732/2);

\draw (4.5/2, 6.062/2)--(1/2,0);

\draw (3/2,0) -- (5.5/2, 4.33/2);

\draw (5/2,0) -- (6.5/2, 2.598/2);

\draw (7/2,0) -- (7.5/2, 0.866/2);

\draw (1/2,0) -- (0.5/2, 0.866/2);

\draw (3/2,0) -- (1.5/2, 2.598/2);

\draw (5/2,0) -- (2.5/2, 4.33/2);

\draw (3.5/2, 6.062/2)--(7/2,0);

\draw (3.5/2,  6.062/2) -- (4.5/2,  6.062/2); 

\draw (2.5/2, 4.33/2) -- (5.5/2, 4.33/2);

\draw (1.5/2, 2.598/2) --  (6.5/2, 2.598/2);

\draw (0.5/2, 0.866/2) -- (7.5/2, 0.866/2);

\draw (4/2, 3.7527/2) -- (4.25/2, 4.1857/2) -- (3.75/2,  4.1857/2)--  (4/2, 3.7527/2);

\node [below] at (4/2, -0.5/2) {$\overline{A_J}=\cup_\beta A_\beta$};

\draw (8, 1.732) -- (6, 1.732) -- (7, 0) --  (8, 1.732) ; 

\draw (7, 0.5774) -- (7.5, 1.4434) -- (6.5,  1.4434)--  (7, 0.5774);

\node [below] at (7, -0.5/2) {$A'_{\beta}\subset A_\beta$};

\end{tikzpicture}
\caption{The nodal set of $\delta_J(N\cdot)$ on $A_J$}
\label{Abeta}
\end{figure}

Now we deal with $\delta_J(NH)$. 
$\overline{A_J}$ can be thought of as an alcove in the $(r-|J|)$-dimensional affine subspace $\bigcap_{j\in J}\mathfrak{t}_{\alpha_j,\delta_{0j}}$ of $\mathfrak{t}$, cut out by the hyperplanes $\mathfrak{t}_{\alpha,n}$ of $\mathfrak{t}$ where $\alpha\in \Sigma\setminus \Sigma_J$ and $n\in\mathbb{Z}$. 
Now all the hyperplanes 
$$\{H\in\mathfrak{t}:\ \alpha(NH)/2\pi i+n=0\}$$ 
of $\mathfrak{t}$ where $\alpha\in\Sigma\setminus \Sigma_J$ and $n\in\mathbb{Z}$, cut $\overline{A_J}$ further into tiny alcoves $A_\beta$ (all identical to $A_J$ in shape) of scale $\asymp N^{-1}$; see Figure \ref{Abeta}. Shrinking each alcove $A_\beta$ to its center by half of its size (say) into $A'_\beta$, it then holds that 
$$|\delta_J( NH)|\gtrsim 1$$
for $H$ lying in each $A'_\beta$. Thus we have 
\begin{align}\label{shrink}
\left\| \frac{\delta_J( N\cdot)}{\delta_J( \cdot)}\right\|_{L^p(S_N)}
&\gtrsim \left(\sum_{\beta}\int_{A'_\beta\cap S_N} t_{j_1}^{-q_{1,0}p}\cdots t_{j_{k}}^{-q_{k,0}p}\ dt_{j_{1}}\cdots\ dt_{j_k}\right)^{\frac{1}{p}}.
\end{align}
As $t_{j_i}>N^{-1}$ on $S_N$ ($i=1,\ldots,k$) and $A_\beta$ is of scale $\asymp N^{-1}$, the values of each $t_{j_i}$ on every  $A_\beta\cap S_N$ are comparable. Moreover, every pair $A_\beta\cap S_N$ and $A_{\beta}'\cap S_N$ are comparable in size. So we can further estimate 
\begin{align*}
\left\| \frac{\delta_J( N\cdot)}{\delta_J( \cdot)}\right\|_{L^p(S_N)}
&\gtrsim \left(\sum_{\beta}\int_{A_\beta\cap S_N} t_{j_1}^{-q_{1,0}p}\cdots t_{j_{k}}^{-q_{k,0}p}\ dt_{j_{1}}\cdots\ dt_{j_k}\right)^{\frac{1}{p}}  \\
& \ \ \ \ \ \ =\left(\int_{S_N} t_{j_1}^{-q_{1,0}p}\cdots t_{j_{k}}^{-q_{k,0}p}\ dt_{j_{1}}\cdots\ dt_{j_k}\right)^{\frac{1}{p}}\asymp\left\| \frac{1}{\delta_J( \cdot )}\right\|_{L^p(S_N)}.
\end{align*}
Thus just as $\left\| 1/{\delta_J( \cdot )}\right\|_{L^p(S_N)}$, 
$\left\| {\delta_J( N\cdot)}/{\delta_J(\cdot)}\right\|_{L^p(S_N)}$ also saturates the bound on the right side of \eqref{keybound}. With \eqref{CSharp} this implies that $\|\chi_\mu\|_{L^p(S_N)}$ for $\mu=N\rho$ indeed saturates the bound in Theorem \ref{characterrestriction}. 
\end{proof}

\begin{rem}\label{varymu}
Choosing $\mu=N\rho$ in $\chi_\mu$ makes the computation much simpler than otherwise. For general $\mu$, as a consequence of Lemma \ref{lemdecomposition}, we still have a formula for restriction of characters to facets of the alcove: 
For any $H\in A_J$, it holds 
\begin{align*}
\chi_\mu(\exp H)=\frac{e^{\rho_J(H_0)}}{\delta_J( H)}\sum_{s\in W_J\backslash W}\det s\ e^{(s\mu)(H+H_0)}\cdot \frac{\prod_{\alpha\in\Sigma^+_J}(\alpha,s\mu)}{\prod_{\alpha\in\Sigma^+_J}(\alpha,\rho_J)}
\end{align*}
where $H_0$ is the unique vector in $\mathfrak{t}^J$ obeying $t_j(H_0)=0$ for all $j\in J$, and $\rho_J=(\sum_{\alpha\in\Sigma_J^+}\alpha)/2$ is the Weyl vector for $\Sigma_J^+$. We conjecture that for any sequence of $\mu$  such that $|\mu|\to\infty$ and  $|\langle\mu,\alpha\rangle|\gtrsim |\mu|$ for all $\alpha\in\Sigma$, 
$\|\chi_\mu\|_{L^p(A_J)}$ also saturates the bound. This is intimately related to the nodal set of the sum in the above character formula. 
The zeros of characters seem to be an intricate and underexplored subject; see \cite{Gal65, Ser04, Pra16a, Pra16b, Ser23}.
The study of nodal sets of general Laplace--Beltrami eigenfunctions is also an attractive subject; see Chapter 13 of \cite{Zel17} for a recent survey. 
\end{rem}

\section{Proof of Theorem \ref{jointrestriction}}
\label{application}
We refer to Chapter 5 of \cite{Hel00} for basic information about the characters and matrix coefficients of irreducible representations of a compact Lie group. 
Any sum $\psi$ of matrix coefficients of the irreducible representation of $U$ of highest weight $\mu-\rho$ ($\mu\in\Lambda^+$) is of Laplace--Beltrami eigenvalue $-|\mu|^2+|\rho|^2=-N^2$. Moreover, it holds 
$$\psi=\psi* (d_\mu\chi_\mu)$$
where $d_\mu$ denotes the dimension of the irreducible representation, and $*$ denotes convolution on the group $U$. We will use the dimension bound 
\begin{align}\label{dimensionbound}
d_\mu=\frac{\prod_{\alpha\in\Sigma^+}(\alpha,\mu)}{\prod_{\alpha\in\Sigma^+}(\alpha,\rho)}\lesssim N^{\frac{d-r}{2}}.
\end{align} 

Let $S$ be a compact submanifold of a maximal flat in $U$. As the the space of matrix coefficients of an irreducible representation is invariant under left (and right) translations, we may as well assume that $S$ is a compact submanifold of a maximal torus $T$ of $U$. 
To prove Theorem \ref{jointrestriction}, it suffices to derive the desired bound for the norm of the operator $\mathcal{T}:L^2(U)\to L^p(S)$ defined by 
\begin{align*}
(\mathcal{T}f)(x):=(f*(d_\mu \chi_\mu))(x)
=\int_{U} f(u) d_\mu\chi_\mu(u^{-1}x)\ du.
\end{align*}

Let $\mathcal{T}^*:L^{p'}(S)\to L^2(U)$ be the dual of $\mathcal{T}$. A direct computation shows that the operator $\mathcal{T}\mathcal{T}^*: L^{p'}(S)\to L^p(S)$ is given by the formula 
\begin{align*}
\mathcal{T}\mathcal{T}^*g(x)&=\int_S g(y)\mathscr{K}(y,x)\ dy
\end{align*}
where 
\begin{align*}
\mathscr{K}(y,x)&=\int_U d_\mu\chi_\mu(u^{-1}x)\overline{d_\mu\chi_\mu(u^{-1}y)} \ du\\
&=\int_U d_\mu\chi_\mu(u^{-1}x) d_\mu\chi_\mu(y^{-1}u) \ du\\
&=(d_\mu\chi_\mu) * (d_\mu \chi_\mu)(x^{-1}y)\\ 
&=d_\mu \chi_\mu(x^{-1}y).
\end{align*}
Here we have used the conjugation-invariant property of $\chi_\mu$.

Let $p\geq 2$. We have for any $y\in S$
$$\|\mathscr{K}(y,\cdot)\|_{L^{\frac{p}{2}}(S)}\leq d_\mu\|\chi_\mu\|_{L^\frac{p}{2}(S^{-1}y)} $$
and for any $x\in S$ 
$$\|\mathscr{K}(\cdot,x)\|_{L^{\frac{p}{2}}(S)}\leq d_\mu\|\chi_\mu\|_{L^\frac{p}{2}(x^{-1}S)} $$
where $$S^{-1}y=
\{x^{-1}y\in T:\ x\in S\},$$ 
$$x^{-1}S=\{x^{-1}y\in T:\ y\in S\}$$ are respectively translates by elements of $T$ of the compact smooth $k$-dimensional submanifolds $S^{-1}$ and $S$ of $T$. Here we also used the translation invariance of the Riemannian volume forms on submanifolds.

Now we apply Theorem \ref{characterrestriction}. By \eqref{fewer} and \eqref{pk=}, we have 
$$p_k\leq p_r=\frac{2r}{d-r}\leq 1\leq \frac{p}{2};$$ 
and $p_k=p/2$ can only happen when $p=2$ with $U$ being the three sphere and $S$ being (part of) a large circle. By Theorem \ref{characterrestriction} especially the part concerning uniformity for translates of a fixed submanifold, we have 
\begin{align*}
\sup_{y\in S}\|\chi_\mu\|_{L^\frac{p}{2}(S^{-1}y)}\lesssim N^{\frac{d-r}{2}-\frac{2k}{p}}
\end{align*}
and 
\begin{align*}
\sup_{x\in S}\|\chi_\mu\|_{L^\frac{p}{2}(x^{-1}S)}\lesssim N^{\frac{d-r}{2}-\frac{2k}{p}}
\end{align*}
except when $p=2$, $U$ is the three sphere, and $S$ is a large circle, in which case an extra multiplicative factor of $\log N$ should be added to the above bounds on the right side. Using \eqref{dimensionbound}, we then have 
\begin{align}\label{Ky}
\sup_{y\in S}\|\mathscr{K}(y,\cdot)\|_{L^{\frac{p}{2}}(S)}\lesssim N^{d-r-\frac{2k}{p}}
\end{align}
and 
\begin{align}\label{Kx}
\sup_{x\in S}\|\mathscr{K}(\cdot, x)\|_{L^{\frac{p}{2}}(S)}\lesssim N^{d-r-\frac{2k}{p}}
\end{align}
adding the $\log N$ factor for the exceptional case. 

Now we recall Schur's test as in Lemma 1.11.14 of \cite{Tao10}.
\begin{lem}[Schur's test]
Let $\mathscr{K}:X\times Y\to \mathbb{C}$ be a measurable function obeying the bounds 
$$\|\mathscr{K}(x,\cdot)\|_{L^{q_0}(Y)}\leq B_0$$
for almost every $x\in X$, and 
$$\|\mathscr{K}(\cdot,y)\|_{L^{p_1'}(X)}\leq B_1$$
for almost every $y\in Y$, where $1\leq p_1,q_0\leq\infty$ and $B_0,B_1>0$. Then for every $0<\theta<1$, the integral operator 
$$Tf(y):=\int_X\mathscr{K}(x,y)f(x)\ d\mu(x)$$ 
is well-defined for all $f\in L^{p_\theta}(X)$ and almost every $y\in Y$, and furthermore 
$$\|Tf\|_{L^{q_\theta}(Y)}\leq B_0^{1-\theta}B_1^{\theta} \|f\|_{L^{p_\theta}(X)}.$$
Here we adopt the convention that $p_0:=1$ and $q_1:=\infty$, thus $q_\theta=q_0/(1-\theta)$ and $p_\theta'=p_1'/\theta$. 
\end{lem}

Using this lemma with $\theta=1/2$ and $q_0=p_1'=p/2$, \eqref{Ky} and \eqref{Kx} together imply 
\begin{align*}
\|\mathcal{T}\mathcal{T}^*\|_{L^{p'}(S)\to L^p(S)}\lesssim N^{d-r-\frac{2k}{p}},
\end{align*}
adding the extra $\log N$ factor for the exceptional case. Taking square root of the above bound gives the desired bound for $\|\mathcal{T}\|_{L^2(U)\to L^p(S)}$ and thus Theorem \ref{jointrestriction}.

\section{Proof of Theorem \ref{gef}}\label{proofofgef}
The Peter--Weyl theorem tells that any $f\in L^2(U)$ is an (infinite) sum of matrix coefficients of irreducible representations of $U$, namely, 
$$f=\sum_{\mu\in\Lambda^+} f_\mu$$
where $f_\mu$ is a sum of matrix coefficients of the irreducible representation of $U$ of highest weight $\mu-\rho$ ($\mu\in\Lambda^+$), and we have the orthogonality condition
$$\|f\|^2_{L^2(U)}=\sum_{\mu\in\Lambda^+} \|f_\mu\|_{L^2(U)}^2.$$ 

Now let $f$ be a Laplace--Beltrami eigenfunction of eigenvalue $-N^2$. Then each nonzero $f_\mu$ appearing in the sum is also a Laplace--Beltrami eigenfunction of eigenvalue $-N^2$, for which we have 
$-|\mu|^2+|\rho|^2=-N^2$. Set 
$$\Lambda_N^+:=\left\{\mu\in\Lambda^+: \ -|\mu|^2+|\rho|^2=-N^2\right\}.$$

Let $S$ be a compact smooth $k$-dimensional submanifold of any maximal flat in $U$, where $k=0,1,\ldots,r$. Let $p\geq 2$. For each $x\in U$, by the Cauchy--Schwarz inequality, we have 
$$|f(x)|=\left|\sum_{\mu\in\Lambda_N^+} f_\mu(x)\right|
\leq |\Lambda_N^+|^{1/2} \|f_\mu(x)\|_{l_\mu^2(\Lambda_N^+)}.$$
By the Minkowski inequality, as $p\geq 2$, we then have  
\begin{align}\label{aftermin}
\|f\|_{L^p(S)}
\leq |\Lambda_N^+|^{1/2}
\left\| \|f_\mu\|_{L^p(S)}\right\|_{l^2_\mu(\Lambda_N^+)}.
\end{align}
A standard estimate of $|\Lambda_N^+|$ is in order. 

\begin{lem}\label{counting}
We have $|\Lambda_N^+|\lesssim N^{r-2}$ for $r\geq 5$, and 
$|\Lambda_N^+|\lesssim_{\varepsilon>0} N^{r-2+\varepsilon}$ for $2\leq r\leq 4$. 
\end{lem}
\begin{proof}
We are counting the number of ways of representing $N^2+|\rho|^2$ by $|\mu|^2$, which is a positive definite quadratic form of rational coefficients in $\mu\in\Lambda^+\subset \Lambda\cong \mathbb{Z}^r$. The estimate is classical; see for example Lemma 23 of \cite{Zha21} for a detailed exposition. 
\end{proof}

By Theorem \ref{jointrestriction}, we have 
$$\|f_\mu\|_{L^p(S)}\lesssim N^{\frac{d-r}{2}-\frac{k}{p}}\|f_\mu\|_{L^2(U)}.$$
Apply this and Lemma \ref{counting} to \eqref{aftermin}, we then have 
\begin{align*}
\|f\|_{L^p(S)}\lesssim_\varepsilon N^{\frac{r-2}{2}+\frac{d-r}{2}-\frac{k}{p}+\varepsilon}\left\| \|f_\mu\|_{L^2(U)} \right\|_{l^2_\mu(\Lambda_N^+)}=N^{\frac{d-2}{2}-\frac{k}{p}+\varepsilon} \|f\|_{L^2(U)},
\end{align*}
with the $\varepsilon$ removable whenever $r\geq 5$. This finishes the proof of Theorem \ref{gef}. 

\begin{rem}\label{rem: sharp?}
    For sharpness of Theorem \ref{gef}, 
    let $\widetilde{\Lambda}_N^+:=\{\mu\in \Lambda_N^+: \langle \mu, \alpha\rangle\gtrsim N,\ \forall \alpha\in\Sigma^+\}$, and consider  
$$f=\sum_{\substack{\mu\in\widetilde{\Lambda}_N^+}}\chi_\mu.$$
Here the regularity condition $\langle \mu, \alpha\rangle\gtrsim N$, $\forall\alpha\in\Sigma^+$ makes sure that when evaluating at the identity $e$ of $U$, it holds $\chi_\mu(e)=d_\mu\asymp\prod_{\alpha\in\Sigma^+}\langle\mu,\alpha\rangle \gtrsim N^{(d-r)/2}$.
By the standard derivative bound for characters,  
    $$\|D_X \chi_\mu\|_{L^\infty(U)} \lesssim d_\mu \|\mu\|\|X\|, \ X\in\mathfrak{u},\ \mu\in\Lambda^+,$$
it follows that 
$|\chi_\mu(u)-d_\mu|\ll d_\mu$, whenever the distance $d(u,e)\ll N^{-1}$ and $\mu\in\Lambda_N^+$.  Now take any $k$-dimensional submanifold $S$ containing the origin $e$ of $U$. The above estimates imply 
$$\frac{\|f\|_{L^p(S)}}{\|f\|_{L^2(U)}}\gtrsim \frac{N^{\frac{d-r}{2}-\frac kp}|\widetilde{\Lambda}_N^+|}{|\widetilde{\Lambda}_N^+|^{\frac12}}=N^{\frac{d-r}{2}-\frac kp}|\widetilde{\Lambda}_N^+|^{\frac12}.$$
For $r\geq 5$, by the equidistribution of lattice points on $r$-dimensional ellipsoids as established by Pommerenke \cite{Pom59}, it holds 
$|\widetilde{\Lambda}_N^+| \gtrsim N^{r-2}$ whenever $\Lambda_N^+\neq\emptyset$, which yields the sharpness of \eqref{eq: gef}. For $r=2,3,4$, the estimate \eqref{eq: gef 234} is essentially sharp up to $N^\varepsilon$-factors, by similar equidistribution results \cite{Pom59, DS90, EH99}. 
\end{rem}

\section{Torus-generated conjugation-invariant  submanifolds}
\label{invariantsub}
Consider the conjugation action of $U$ on each facet $A_J$, $J\subsetneqq\{0,1,\ldots,r\}$. 
Let 
$$\mathfrak{t}_J:=\{H\in\mathfrak{t}:\ t_{j}(H)=0\text{ for all }j\in J\}$$
\footnote{The $\mathfrak{t}^{J\perp}$ defined in \eqref{tJperp} is the linear subspace associated to the affine subspace $\mathfrak{t}_J$ of $\mathfrak{t}$.}and consider the pointwise stabilizer subgroup
$$T_J=\{u\in U: \ \exp \text{Ad}(u) H=\exp H\text{ for all }H\in \mathfrak{t}_J\}$$
of $U$. We have: 

\begin{lem}\label{dimTJ}
$\dim T_J=r+2|\Sigma_J^+|$.
\end{lem}
\begin{proof}
This is Lemma 5.1 in Ch. VII of \cite{Hel01}. 
\end{proof}

Consider the mapping 
$$\Psi_J: (uT_J,H)\mapsto \exp \text{Ad}(u) H$$
of $(U/T_J)\times A_J$ into $U$. 

\begin{lem}
The smooth mapping $\Psi_J:(U/T_J)\times A_J\to U$ is an immersion. Moreover, let $du$, $d(uT_J)$, and $dH$ denote the volume form on $U$, $U/T_J$, and $A_J$ respectively, all canonically induced from the Riemannian metric on $U$. Then the pullback $\Psi_J^*(du)$ of $du$ by $\Psi_J$ equals $C  |\delta_J( H)|^2\ d(uT_J)\ dH$, where $C$ is a positive constant. 
\end{lem}
\begin{proof}
We omit the details as it follows from Lemma 5.2 and its proof in Ch. VII of \cite{Hel01}. 
\end{proof}

Let $U_J$ denote the image of the mapping $\Psi_J:(U/T_J)\times A_J\to U$. Even though $U_J$ may not be strictly a Riemannian manifold by itself as it might have self-intersections, $U_J$ is understood as an immersed Riemannian submanifold of $U$ equipped with a canonical measure still denoted by $du$ which equals the pushforward of $C |\delta_J( H)|^2\ d(uT_J)\ dH$ by $\Psi_J$, so that the following singular analogues of the Weyl integration formula hold
$$\int_{U_J}f(u)\ du=C \int_{(U/T_J)\times A_J} f(\exp \text{Ad}(u) H)|\delta_J( H)|^2\ d(uT_J)\ dH.$$
Specializing $f$ to be conjugation-invariant functions, we record this identity as the following lemma. 

\begin{lem}\label{lem: adapted Weyl formula}
Suppose $f$ is a smooth function on $U$ that is conjugation-invariant. Then 
$$\int_{U_J}f(u)\ du=C\int_{A_J} f(\exp H)|\delta_J( H)|^2\ dH.$$
\end{lem}

More generally, we may replace each facet $A_J$ by any of its submanifold $S$, and consider the conjugation-action of $U$ on $S$. 

\begin{defn}\label{definv}
Let $k=0,1,\ldots,r$. 
We say $Y$ is a torus-generated conjugation-invariant submanifold of $U$ of rank $k$, if there is a smooth $k$-dimensional submanifold $S$ of a facet $A_J$, such that $Y$ equals the image of the smooth mapping 
$$\Psi_S: (uT_J,H)\mapsto \exp \text{Ad}(u) H$$
of $(U/T_J)\times S$ into $U$. 
\end{defn}

Completely analogously, 
we have the following lemmas. 
 
\begin{lem}\label{PsiS}
The smooth mapping $\Psi_S:(U/T_J)\times S\to U$ is an immersion. Moreover, let $du$, $d(uT_J)$, and $dH$ denote the volume form on $U$, $U/T_J$, and $S$ respectively, all canonically induced from the Riemannian metric on $U$. Then the pullback $\Psi_S^*(du)$ of $du$ by $\Psi_S$ equals $C |\delta_J( H)|^2\ d(uT_J)\ dH$ where $C$ is a positive constant. Then as an immersed Riemannian submanifold of $U$, $Y$ is equipped with a canonical measure which equals the pushforward of $C |\delta_J( H)|^2\ d(uT_J)\ dH$ by $\Psi_S$.
\end{lem}

\begin{lem} \label{Weylintegration}
Suppose $f$ is a smooth function on $U$ that is conjugation-invariant. Then 
$$\int_{Y}f(u)\ du=C\int_{S} f(\exp H)|\delta_J( H)|^2\ dH.$$
\end{lem}

As a consequence of Lemma \ref{dimTJ} and \ref{PsiS}, we have 
\begin{align}\label{dimY}
\dim Y=k+d-r-2|\Sigma^+_J|=k+2(|\Sigma^+|-|\Sigma^+_J|).
\end{align}

\section{Proof of Theorem \ref{singular} and \ref{singular2}}
\label{invariant}
Recall that for $\mu\in\Lambda^+$, the Laplace--Beltrami eigenvalue of $\chi_\mu$ equals $-|\mu|^2+|\rho|^2=-N^2$. As a torus-generated conjugation-invariant  submanifold of $U$ of rank $k$, by Definition \ref{definv}, $Y$ equals the image of the smooth mapping 
$$\Psi_S: (uT_J,H)\mapsto \exp \text{Ad}(u) H$$
of $(U/T_J)\times S$ into $U$, where
$S$ is a smooth $k$-dimensional submanifold of a facet $A_J$, and $T_J$ is the corresponding pointwise stabilizer subgroup. By Lemma \ref{Weylintegration}, we have 
$$\|\chi_\mu\|^p_{L^p(Y)}=\int_Y|\chi_\mu(u)|^p \ du=C\int_S |\chi_\mu(\exp H)|^p\cdot |\delta_J( H)|^2\ dH
=C\left\|\chi_\mu\cdot |\delta_J|^{\frac{2}{p}}\right\|^{p}_{L^p(S)}.$$
Using the barycentric-semiclassical subdivision again, it suffices to estimate the above integral replacing $S$ by each 
$$S_{K,J'}:=S\cap P_{K,J'}$$
where $J'\subset K\subsetneqq\{0,\ldots,r\}$. As $S\subset A_J$, we may assume that $J\subset J'$. 
Using \eqref{char}, we have
$$|\chi_\mu(\exp H)|\cdot |\delta_J( H)|^{\frac{2}{p}}
\leq\frac{|\delta_J^{J'}( H)|^{\frac{2}{p}}}{|W_{J'}|\cdot |\delta_{J'}( H)|^{1-\frac{2}{p}}}\sum_{s\in W}\left|\chi^{J'}_{(s\mu)^{J'}}(\exp H^{J'})\right|.$$
Note that for $H\in S_{K,J'}\subset P_{J'}$, 
$$|\delta_J^{J'}( H)|=\prod_{\alpha\in\Sigma_{J'}^+\setminus \Sigma_J^+}\left|e^{\frac{\alpha(H)}{2}}-e^{-\frac{\alpha(H)}{2}}\right|\lesssim N^{-|\Sigma_{J'}^+|+|\Sigma_J^+|}.$$
Now Lemma \ref{charbound} gives that for $H\in S_{K,J'}\subset P_{J'}$, 
$$|\chi_\mu(\exp H)|\cdot |\delta_J( H)|^{\frac{2}{p}}\lesssim N^{-\frac{2}{p}(|\Sigma_{J'}^+|-|\Sigma_J^+|)+|\Sigma^+_{J'}|}\cdot\frac{1}{|\delta_{J'}(H)|^{1-\frac{2}{p}}}.$$
Hence 
$$\left\|\chi_\mu\cdot |\delta_J|^{\frac{2}{p}}\right\|_{L^p(S_{K,J'})}
\lesssim N^{-\frac{2}{p}(|\Sigma_{J'}^+|-|\Sigma_J^+|)+|\Sigma^+_{J'}|}\cdot\left\|\frac{1}{\delta_{J'}}\right\|_{L^{p-2}(S_{K,J'})}^{1-\frac{2}{p}}.$$
Then the desired bounds follow from Proposition \ref{keyprop}, noting the dimension formula \eqref{dimY}. Note that the $p=2$ case does not require Proposition \ref{keyprop}. 

Taking orbits of the facets which have been shown in Section \ref{sharpness} to saturate the bounds in Theorem \ref{characterrestriction}, we get the torus-generated conjugation-invariant  submanifolds for which the bounds in Theorem \ref{singular} can be saturated. Namely, for each $k=0,1,\ldots,r$, using Lemma \ref{nk}, for any permutation $\mathcal{P}_0=(j_0,j_1,\ldots,j_r)$ of $\{0,1,\ldots,r\}$ that realizes $q_i(\mathcal{P}_0)=q_{i,0}$ for all $i=0,1,\ldots,r$, set $J:=\{j_{k+1},\ldots,j_{r}\}$. Let $Y$ be the image of the mapping $\Psi_J: (uT_J,H)\mapsto \exp \text{Ad}(u) H$
of $(U/T_J)\times A_J$ into $U$. Using \eqref{CSharp} and the adapted Weyl integration formula as in Lemma \ref{Weylintegration}, an entirely similar computation as in Section \ref{sharpness} shows that $\|\chi_\mu\|_{L^p(Y)}$ saturates the bounds for all $p\geq 2$, where we still let $\mu=N\rho$ where $N$ is a natural number growing to infinity. 

Lastly, we show that the bound of the $p=2$ case can actually be saturated on orbits of any facets. 
For any facet $A_J$ of the alcove $A$, still let $Y$ be the image of the mapping $\Psi_J: (uT_J,H)\mapsto \exp \text{Ad}(u) H$
of $(U/T_J)\times A_J$ into $U$. Still for $\mu=N\rho$ where $N$ is any (large) natural number, using \eqref{CSharp} and Lemma \ref{Weylintegration}, we have 
$$\|\chi_\mu\|_{L^2(Y)}\asymp N^{|\Sigma^+_J|}\|\delta_J( N\cdot)\|_{L^2(A_J)}=N^{\frac{d-r}{2}-\frac{n-k}{2}}\|\delta_J( N\cdot)\|_{L^2(A_J)}.$$
By an argument entirely similar to \eqref{shrink}, we get 
$\|\delta_J( N\cdot)\|_{L^2(A_J)}\gtrsim 1$. Hence 
$\|\chi_\mu\|_{L^2(Y)}\gtrsim N^{(d-r)/2-(n-k)/2}$. 
Thus the bound for the $p=2$ case is indeed sharp for any such submanifold $Y$.

\end{document}